\documentclass[letterpaper]{article}
\usepackage[margin=1in]{geometry}
\usepackage{amsmath, amssymb,amsthm, xcolor}
\usepackage{mathabx}
\usepackage{bbm}
\usepackage{tikz}
\usepackage{bm}
\usepackage{float}
\usepackage{scalerel}
\usepackage[title]{appendix}

\title{Fluctuation exponents for stationary exactly solvable lattice polymer models via a Mellin transform framework}
\author{Hans Chaumont\thanks{Department of Mathematics, University of Wisconsin. Madison, WI, USA.\newline \href{mailto:chaumont@wisc.edu}{chaumont@wisc.edu}} \and Christian Noack\thanks{Department of Mathematics, University of Wisconsin. Madison, WI, USA.\newline \href{mailto:cnoack@wisc.edu}{cnoack@wisc.edu}}}
\usepackage{hyperref}
\newtheorem{theorem}{Theorem}[section]
\newtheorem{corollary}[theorem]{Corollary}
\newtheorem{lemma}[theorem]{Lemma}
\newtheorem{definition}[theorem]{Definition}

\newtheorem{proposition}[theorem]{Proposition}

\theoremstyle{remark}
\newtheorem{remark}[theorem]{Remark}
\newtheorem{hypothesis}[theorem]{Hypothesis}

\numberwithin{equation}{section}

\newcommand{\p}{\partial}
\newcommand{\Z}{\mathbb{Z}}
\newcommand{\C}{\mathbb{C}}

\newcommand{\N}{\mathbb{N}}
\newcommand{\R}{\mathbb{R}}

\newcommand{\mbb}{\mathbb}
\newcommand{\ind}{\mathbbm{1}}
\newcommand{\E}{\mathbb{E}}

\newcommand{\lam}{\lambda}

\newcommand{\Var}{\mathbb{V}\text{ar}}
\renewcommand{\t}{\widetilde}
\newcommand\spt{\text{supp}}
\newcommand{\Eh}{\widehat{\E}}

\newcommand{\ta}{\t{a}}

\renewcommand{\hat}{\widehat}
\renewcommand{\tilde}{\widetilde}

\newcommand\floor[1]{\lfloor#1\rfloor}

\renewcommand{\P}{\mathbb{P}}

\newcommand\abullet{{\scaleobj{0.6}{\bullet}}}

\newcommand{\Cov}{\C\text{ov}}
\DeclareMathOperator{\supp}{supp}

\renewcommand{\centerdot}{\abullet}

\begin{document}
\maketitle

\begin{abstract}
We develop a Mellin transform framework which allows us to simultaneously analyze the four known exactly solvable $1+1$ dimensional lattice polymer models: the log-gamma, strict-weak, beta, and inverse-beta models. Using this framework we prove the conjectured fluctuation exponents of the free energy and the polymer path for the stationary point-to-point versions of these four models. The fluctuation exponent for the polymer path was previously unproved for the strict-weak, beta, and inverse-beta models.
\end{abstract}



\section{Introduction}
The directed polymer in a random environment was first introduced by Huse and Henley \cite{HH1985} to model the interaction between a long chain of molecules and random impurities. This was later reformulated by Imbrie and Spencer \cite{IS1988} as a random walk in a random environment. See the recent lectures notes by Comets \cite{Comets2017} for additional historical background and a survey of techniques used to study directed polymers. In the $1+1$ dimensional case, a large class of polymer models are expected to lie in the KPZ universality class. For this class, the polymer path and free energy fluctuation exponents are conjectured to be $2/3$ and $1/3$, respectively, and the probability distribution of the rescaled free energy is conjectured to converge to the Tracy-Widom GUE distribution.

There are a few $1+1$ dimensional polymer models for which these results have been proved. Bal\'azs, Quastel, and Sepp\"al\"ainen \cite{BQS2011}, prove the fluctuation exponents for a Hopf-Cole solution to the KPZ equation with Brownian initial condition. This solution can be interpreted as the free energy of a stationary continuum directed polymer.
Amir, Corwin, and Quastel \cite{ACQ2011} study the Hopf-Cole solution to the KPZ equation with narrow-wedge initial condition and prove Tracy-Widom limit distribution for large time. For the O'Connell-Yor semi-discrete Brownian directed polymer \cite{OY2001}, the fluctuation exponents are proved by Sepp\"al\"ainen and Valk\'o \cite{SV2010}, and the limit distribution is proved in Borodin, Corwin  \cite{BC2014} and  Borodin, Corwin, Ferrari \cite{BCF2014}. 

In the setting of lattice directed polymers, there are four models for which results about the scaling exponents or limit distributions are known. The log-gamma directed polymer was introduced by Sepp\"al\"ainen in \cite{S2012}, where the fluctuation exponents were proved. The limit distribution result was proved by Borodin, Corwin, and Remenik \cite{BCR2013}. The strict-weak polymer model was introduced by Sepp\"al\"ainen and its limit distribution was proved simultaneously through different methods by Corwin, Sepp\"al\"ainen, Shen in \cite{CSS2015} and O'Connell, Ortmann in \cite{OO2015}. The beta directed polymer was introduced by Barraquand and Corwin in \cite{BC2014}, where its limit distribution was also calculated. The fourth model is the inverse-beta model, introduced by Thiery and Le Doussal in \cite{TL2015}, in which they conjecture a formula for the Laplace transform of the polymer partition function and, contingent on this conjecture, show Tracy-Widom limit distribution for the rescaled free energy.

In this paper we provide a Mellin transform framework with which we are able to treat these four lattice polymer models simultaneously and prove the fluctuation exponents of the free energy and the polymer path for their stationary versions. While for the log-gamma model these results were previously shown \cite{S2012}, for the strict-weak, beta, and inverse-beta models, the path fluctuation results are new.

Our methods rely upon a Burke-type stationarity property that each of these models possesses. This stationarity, along with a coupling argument, are used to prove a variance formula which is then amenable to analysis. This method was first used by Cator and Groeneboom \cite{CG2006} to prove the order of the variance of the length of the longest weakly North-East path in Hammersley's process with sources and sinks. In \cite{BCS2006}, Bal\'azs, Cator, and Sepp\"al\"ainen adapt this method to prove the order of the fluctuations of the passage time and the fluctuations of the maximal path for last passage percolation with exponential weights. Sepp\"al\"ainen \cite{S2012} used this method to prove the order of the fluctuation of the free energy and the polymer path fluctuations for the point-to-point log-gamma model with stationary boundary conditions, and upper bounds on the fluctuation exponents for the non-stationary point-to-point and point-to-line models. Sepp\"al\"ainen and Valk\'o \cite{SV2010} prove the scaling exponents for the O'Connell-Yor polymer, and Flores, Sepp\"al\"ainen, and Valk\'o \cite{FSV2014} extend the result to the intermediate disorder regime. Our paper closely follows the methods in \cite{S2012}.

In our related paper \cite{CN2017} we prove that in the setting of $1+1$ dimensional lattice directed polymers, the only four models possessing the Burke-type stationarity property are the log-gamma, strict-weak, beta, and inverse-beta models.
\\

Notation: $\N= \{1,2,\ldots\}$, $\Z_+ = \{0,1,\ldots\}$, and $\R$ denotes the real numbers. Let $\floor{x}$ denote the greatest integer less than or equal to $x$. Let $\vee$ and $\wedge$ denote maximum and minimum, respectively. Given a real valued function $f$, let $\supp(f)= \{x: f(x)\neq 0\}$ denote the support of the function $f$ (note that we do not insist on taking the closure of this set). Given a random variable $X$ with finite expectation, we let $\overline{X}= X-\E[X]$.  For $A\subset \R$ write $-A=\{ -a : a\in A\}$ and $A^{-1}=\{a^{-1}: a\in A \}$ assuming that $0\notin  A$. The symbol $\otimes$ is used to denote (independent) product distribution.

\subsection{The polymer model}

On each edge $e$ of the $\Z_+^2$  lattice we place a positive random weight. The superscripts 1 and 2 will be used to denote horizontal and vertical edge weights, respectively. For $z\in \N^2$, let $Y^1_z$ and $Y^2_z$ denote the horizontal and vertical incoming edge weights. We assume that the collection of pairs $\{(Y^1_z,Y^2_z)\}_{z\in \N^2}$ is independent and identically distributed with common distribution $(Y^1,Y^2)$, but do not insist that $Y^1_z$ is independent of $Y^2_z$. Call this collection the \emph{bulk weights}. For $x\in \N\times \{0\}$, let $R^1_x$ denote the horizontal incoming edge weight, and for $y\in \{0\}\times \N$, let $R^2_y$ denote the vertical incoming edge weight. We assume the collections $\{R^1_x\}_{x\in \N\times \{0\}}$ and $ \{R^2_y\}_{y\in \{0\}\times \N}$ are independent and identically distributed with common distributions $R^1$ and $R^2$, and refer to them as the \emph{horizontal} and \emph{vertical} \emph{boundary weights}, respectively. We further assume that the horizontal boundary weights, the vertical boundary weights, and the bulk weights are independent of each other. This assignment of edge weights is illustrated in Figure \ref{fig-weights}. We call 
\begin{equation}\label{environment}
\omega = \{R^1_x, R^2_y, (Y^1_z,Y^2_z): x\in \N\times \{0\}, y\in \{0\}\times \N, z\in \N^2\}
\end{equation}
the \emph{polymer environment}. We use $\P$ and $\E$ to denote the probability measure and corresponding expectation of the polymer environment.

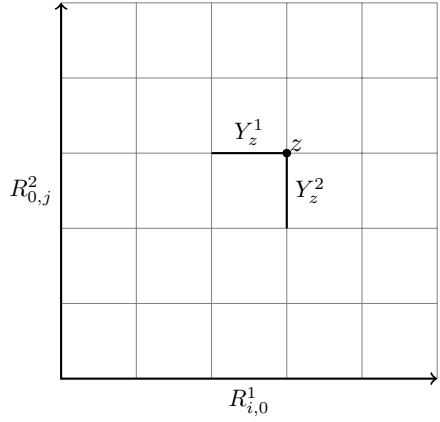
\begin{figure}[ht]
  \centering
  \begin{tikzpicture}
    
   \draw [help lines] (0,0) grid (5,5);	  
   \draw [thick, <->] (0,5) -- (0,0) -- (5,0);
   \draw [thick] (2,3) -- (3,3) -- (3,2);
   \node [left,scale=.9] at (0,2.5) {$R^2_{0,j}$};
   \node [below,scale=.9] at (2.5,0) {$R^1_{i,0}$};
   \node [above,scale=.9] at (2.5,3) {$Y^1_z$};
   \node [right,scale=.9] at (3,2.5) {$Y^2_z$};
   \draw[fill] (3,3) circle [radius=.05];
   \node [above right] at (2.91,2.91) {$z$};

  \end{tikzpicture}
    \caption{Assignment of edge weights.}
  \label{fig-weights}
\end{figure}

A path is weighted according to the product of the weights along its edges. For $(m,n)\in
\Z_+^2\setminus \{(0,0)\}$ we define a probability measure on all up-right paths from
$(0,0)$ to $(m,n)$. See Figure \ref{up-right path} for an example of an up-right path. Let $\Pi_{m,n}$ denote the collection of all such
paths. We identify paths $x_\abullet = (x_0, x_1, \ldots, x_{m+n})$ by
their sequence of vertices, but also associate to paths their sequence
of edges $(e_1, \ldots, e_{m+n})$, where $e_i = \{x_{i-1}, x_i\}$. Define the quenched polymer measure on $\Pi_{m,n}$,
\[ 
Q_{m,n} (x_\abullet) := \frac{1}{Z_{m,n}}
\prod_{i=1}^{m+n} \omega_{e_i}, 
\] 
where $\omega_e$ is the weight associated to the edge $e$ and
\[ 
Z_{m,n} := \sum_{x_\abullet \in \Pi_{m,n}} \prod_{i=1}^{m+n}
\omega_{e_i} 
\] 
is the associated partition function. At the origin, define $Z_{0,0}:=1$. Taking the expectation $\E$ of the quenched measure with respect to the edge weights gives the annealed measure on $\Pi_{m,n}$,
\[
P_{m,n}(x_\centerdot) :=\E[Q_{m,n}(x_\centerdot)].
\]
The annealed expectation will be denoted by $E_{m,n}$.

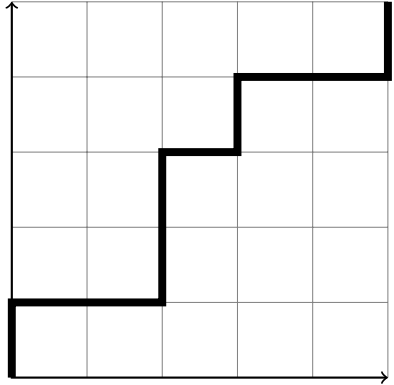
\begin{figure}[ht]
\centering
	\begin{tikzpicture}
  		\draw [help lines] (0,0) grid (5,5);	  
   		\draw [thick, <->] (0,5) -- (0,0) -- (5,0);
   		\draw [line width=3pt] (0,0) -- (0,1) -- (1,1) -- (2,1) -- (2,2) -- (2,3) -- (3,3) -- (3,4) --(4,4) -- (5,4) -- (5,5);
	\end{tikzpicture}
	\caption{An up-right path from $(0,0)$ to $(5,5)$.}
    \label{up-right path}
\end{figure}
We specify the edge weight distributions for the four stationary polymer models.
The notation $X\sim \text{Ga}(\alpha,\beta)$ is used to denote that a random variable is gamma$(\alpha,\beta)$ distributed, i.e.\ has density $\Gamma(\alpha)^{-1} \beta^\alpha x^{\alpha-1}e^{-\beta x}$ supported on $(0,\infty)$, where $\Gamma(\alpha) = \int_0^\infty x^{\alpha-1} e^{-x} dx$ is the gamma function. $X\sim \text{Be}(\alpha,\beta)$ is used to say that $X$ is beta$(\alpha,\beta)$ distributed, i.e.\ has density $\frac{\Gamma(\alpha+\beta)}{\Gamma(\alpha) \Gamma(\beta)} x^{\alpha-1}(1-x)^{\beta-1} $ supported on $(0,1)$. We then use $X\sim \text{Ga}^{-1}(\alpha,\beta)$ and $X\sim \text{Be}^{-1}(\alpha,\beta)$ to denote that $X^{-1}\sim \text{Ga}(\alpha,\beta)$ and $X^{-1}\sim \text{Be}(\alpha,\beta)$, respectively. We also use $X\sim \left(\text{Be}^{-1}(\alpha,\beta) -1\right)$ to denote that $X+1\sim\text{Be}^{-1}(\alpha,\beta)$.

\begin{itemize}
\item \textbf{Inverse-gamma (IG)}: This is also known as the log-gamma model. Assume $\mu>\theta>0, \, \beta>0$ and
\begin{equation}
\begin{gathered}
R^1 \sim \text{Ga}^{-1}(\mu-\theta,\beta)\qquad  R^2\sim \text{Ga}^{-1}(\theta,\beta)\\
(Y^1,Y^2)= (X,X) \qquad \text{where} \qquad X\sim \text{Ga}^{-1}(\mu,\beta).
\end{gathered}\label{model-IG}
\end{equation}
\item \textbf{Gamma (G)}: This is also known as the strict-weak model. Assume $\theta,\,\mu,\,\beta>0$ and 
\begin{equation}
\begin{gathered}
R^1\sim \text{Ga}(\mu+\theta,\beta)\qquad R^2\sim \text{Be}^{-1}(\theta,\mu)\\
(Y^1,Y^2) = (X, 1) \qquad \text{where} \qquad X\sim \text{Ga}(\mu,\beta).
\end{gathered}\label{model-G}
\end{equation}
\item \textbf{Beta (B)}: 
Assume $\theta,\,\mu,\,\beta>0$ and 
\begin{equation}
\begin{gathered}
R^1\sim \text{Be}(\mu+\theta,\beta)\qquad R^2 \sim \text{Be}^{-1}(\theta,\mu)\\
(Y^1,Y^2) = (X,1-X) \qquad \text{where} \qquad X\sim \text{Be}(\mu,\beta).
\end{gathered}\label{model-B}
\end{equation}
\item \textbf{Inverse-beta (IB)}:
Assume $\mu>\theta>0,\, \beta>0$ and 
\begin{equation}
\begin{gathered}
R^1 \sim \text{Be}^{-1}(\mu-\theta,\beta)\qquad R^2 \sim \left(\text{Be}^{-1}(\theta,\beta+\mu-\theta)-1\right)\\
(Y^1,Y^2) = (X, X-1) \qquad \text{where} \qquad X\sim \text{Be}^{-1}(\mu,\beta).
\end{gathered}\label{model-IB}
\end{equation}
\end{itemize}
The name of each model refers to the distribution of the bulk weights. We call these models the \textbf{four basic beta-gamma models}.

\subsection{Results}

If $X$ is a positive random variable with density $\rho$, define
\begin{align}
L_X(x):= -\frac{1}{x\rho(x)} \Cov (\log X, \ind_{\{X\leq x\}})\label{definition L_X}
\end{align}
for all $x$ such that $\rho(x)>0$. Given a path $x_\centerdot\in \Pi_{m,n}$, define the \emph{exit points} of the path from the horizontal and vertical axes by
\begin{align}
t_1 := \max\{i:(i,0)\in x_\centerdot\} \qquad \text{and} \qquad t_2:= \max\{j:(0,j)\in x_\centerdot\}.\label{def exit points}
\end{align}

The following proposition gives exact formulas for the expectation and variance of the free energy, which is a starting point for analysis of these four models.
\begin{proposition}\label{proposition variance formula}
Assume that the polymer environment has edge weight distributions $R^1, R^2, (Y^1,Y^2)$ as in one of \eqref{model-IG} through \eqref{model-IB}.  Then for all $(m,n)\in \Z_+^2$, 
\begin{align}
\E[\log Z_{m,n}] &= m \E[\log R^1] + n \E[\log R^2], \nonumber\\
\Var[\log Z_{m,n}] &= -m \Var[\log R^1] + n \Var[\log R^2] + 2 E_{m,n} \left[\sum_{i=1}^{t_1} L_{R^1}(R^1_{i,0})\right],\label{equation var formula 1}\\
\Var[\log Z_{m,n}] &= m \Var[\log R^1] - n \Var[\log R^2] + 2 E_{m,n}\left[ \sum_{j=1}^{t_2} L_{R^2}(R^2_{0,j})\right].\label{equation var formula 2}
\end{align}
\end{proposition}

Using these exact formulas, we can obtain the following upper bound on the variance of the free energy when $(m,n)$ grow in a characteristic direction. 

\begin{theorem}\label{theorem stationary variance bounds}
Assume that the polymer environment has edge weight distributions $R^1, R^2, (Y^1,Y^2)$ as in one of \eqref{model-IG} through \eqref{model-IB}, and let $(m,n)=(m_N,n_N)_{N=1}^\infty$ be a sequence such that
\begin{align}
|m_N-N \Var[\log R^2]|\leq \gamma N^{2/3} \qquad \text{and} \qquad |n_N-N\Var[\log R^1]|\leq \gamma N^{2/3} \label{direction}
\end{align}
for some fixed $\gamma>0$.  Then there exist positive constants $c,\, C$, and $N_0$ depending only on $\mu,\theta,\beta,\gamma$ such that for all $N\geq N_0$,
\[
c N^{2/3}\leq
\Var[\log Z_{m,n}]\leq C N^{2/3}.
\]
The same constants $c,\, C,\, N_0$ can be taken for all $\mu,\theta,\beta,\gamma$ varying in a compact set.
\end{theorem}
Theorem \ref{theorem stationary variance bounds} and a Borel-Cantelli argument give the following law of large numbers.
\begin{corollary}
With assumptions as in Theorem \ref{theorem stationary variance bounds} the following limit holds $\P$-almost surely
\begin{align*}
\lim_{N\to \infty}\frac{\log Z_{m,n}}{N} = \E[\log R^1] \Var[\log R^2] + \E[\log R^2] \Var[\log R^1].
\end{align*}
\end{corollary}

The following is a result for when the sequence $(m_N,n_N)$ does not satisfy condition \eqref{direction}. The statement is given for when the horizontal direction is too large, but an analogous result holds for the vertical direction.
\begin{corollary}\label{corollary CLT}
Assume that the polymer environment has edge weight distributions $R^1, R^2, (Y^1,Y^2)$ as in one of \eqref{model-IG} through \eqref{model-IB} and that $m,n\to\infty$. Define $N$ by $n= N\Var[\log R^1]$ and assume
\[ N^{-\alpha}(m - N\Var[\log R^2]) \to c_1 >0 \]
for some $\alpha>2/3$. Then as $N\to \infty$
	\[N^{-\alpha/2} \left( \log Z_{m,n} - \E[\log Z_{m,n}] \right)\]
converges in distribution to a centered normal with variance $c_1 \Var[\log R^1]$. 
\end{corollary}

The variance formulas in Proposition \ref{proposition variance formula} connect the variance of the free energy to the exit points of the path from the boundaries \eqref{def exit points}. This allows us to obtain bounds on the polymer path fluctuations under the annealed measure.

Given a path $x_\centerdot\in \Pi_{m,n}$, for $0\leq k\leq m$ and $0\leq l\leq n$ define
\begin{align*}
v_0(l) &:= \min\{i: (i,l)\in x_\centerdot\} & v_1(l) &:= \max\{i:(i,l) \in x_\centerdot\}\\
w_0(k) &:= \min\{j: (k,j) \in x_\centerdot\} & w_1(k) &:= \max\{j:(k,j)\in x_\centerdot\}.
\end{align*}
This is illustrated in Figure \ref{figure vw}.
\begin{figure}[H]
\centering
\begin{tikzpicture}[scale=1.15]
\draw (0,0) -- (6,0) -- (6,5) -- (0,5) -- (0,0);
\draw [ultra thick] (0,2) -- (6,2);
\draw [ultra thick] (4,0) -- (4,5);

\node [below] at (4,0) {$k$};
\node[left] at (0,2) {$l$};

\draw [line width=2.5pt] (0,0) -- (.5,0) -- (.5,.75) -- (.75,.75) -- (.75,1) -- (1.5,1) -- (1.5,1.25) -- (1.75,1.25) -- (1.75,1.75) -- (2,1.75) -- (2,2) -- (2.75,2) -- (2.75,2.5) -- (3.25,2.5) -- (3.25, 3.25) -- (3.75, 3.25) -- (3.75,3.5) -- (4,3.5) -- (4,4.25) -- (4.75,4.25) -- (4.75,4.75) -- (5,4.75) -- (5,5) -- (6,5);

\draw [dashed] (2,0) -- (2,2);
\draw [dashed] (2.75,0) -- (2.75,2);
\draw [dashed] (0,3.5) -- (4,3.5);
\draw [dashed] (0,4.25) -- (4,4.25);

\node [below] at (2.75,0) {$v_1(l)$};
\node [below] at (2,0) {$v_0(l)$};

\node [left] at (0,3.5) {$w_0(k)$};
\node [left] at (0,4.25) {$w_1(k)$};

\end{tikzpicture}
\caption{Example path with $v_0,v_1, w_0,w_1$ illustrated.}
\label{figure vw}
\end{figure}
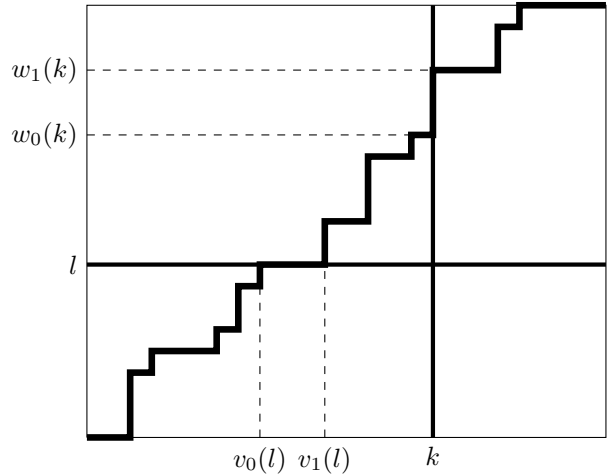
\begin{theorem}\label{theorem stationary PF UB}
Assume that the polymer environment has edge weight distributions $R^1, R^2, (Y^1,Y^2)$ as in one of \eqref{model-IG} through \eqref{model-IB}, and let $(m,n)=(m_N,n_N)_{N=1}^\infty$ be a sequence satisfying \eqref{direction} for some fixed $\gamma>0$. Let $0\leq \tau <1$. Then there exist positive constants $b_0,\, C,\, c_0,\,c_1,\,N_0$ depending only on $\mu,\,\theta,\,\beta,\,\gamma,\,\tau$ such that for $b\geq b_0$ and $N\in \N$,
\begin{align}
P_{m,n}( v_0(\floor{\tau n}) \leq \tau m - bN^{2/3} \text{ or } v_1(\floor{\tau n}) \geq \tau m + bN^{2/3}) \leq \frac{C}{b^3},\label{pf-bound-v}\\
P_{m,n}( w_0(\floor{\tau m}) \leq \tau n - bN^{2/3} \text{ or } w_1(\floor{\tau m}) \geq \tau n + bN^{2/3}) \leq \frac{C}{b^3},\label{pf-bound-h}
\end{align}
and for all $N\geq N_0$,
\begin{equation}\label{pf-lower bound}
c_0\leq P_{m,n}(v_1(\floor{\tau n})\geq \tau m + c_1 N^{2/3} \text{ or } w_1(\floor{\tau m})\geq\tau n +c_1 N^{2/3}).
\end{equation}
The same constants  can be taken for all $\mu,\,\theta,\,\beta,\,\gamma,\,\tau$ varying in a compact set.
\end{theorem}

\emph{Structure of the paper:} In Section \ref{section DRP} we define the down-right property then state and prove consequences of this property. In Section \ref{section Mellin} we introduce the Mellin transform framework, which allows us to treat the four basic beta-gamma models simultaneously, and prove Proposition \ref{proposition variance formula}. In Section \ref{section variance UB} we prove the upper bound of Theorem \ref{theorem stationary variance bounds}. In Section \ref{section PF UB} we prove bounds \eqref{pf-bound-v} and \eqref{pf-bound-h} of Theorem \ref{theorem stationary PF UB}. In Section \ref{section lower bounds} we prove the lower bound of Theorem \ref{theorem stationary variance bounds} and bound \eqref{pf-lower bound} of Theorem \ref{theorem stationary PF UB}. In Appendix \ref{appendix CSS} we verify that each of the four basic beta-gamma models satisfies the conditions of Hypothesis \ref{hypothesis CSS}. Appendix \ref{appendix misc} collects technical lemmas used in Sections \ref{section Mellin} and \ref{section variance UB}. Appendix \ref{appendix psi} collects facts used in the proof of Proposition \ref{proposition lower bound}.\\

\emph{Acknowledgements:} This work is part of our dissertation research at the University of Wisconsin-Madison.  We thank our advisors Timo Sepp{\"a}l{\"a}inen and Benedek Valk\'o for their guidance and insights.

\section{The down-right property}\label{section DRP}
Write $\alpha_1=(1,0)$, $\alpha_2=(0,1)$. For $k=1,2$ define ratios of partition functions
\begin{align*}
R^k_x:=\frac{Z_x}{Z_{x-\alpha_k}}\qquad \text{for all  }x\text{ such that }  x-\alpha_k\in\mathbb{Z}_+^2.
\end{align*}
Note that these extend the definitions of $R^1_{i,0}$ and $R^2_{0,j}$, since for example $Z_{i,0}=\prod_{k=1}^{i} R^1_{k,0}$.
We say that $\pi=\{\pi_k\}_{k\in\mathbb{Z}}$ is a down-right path in $\Z_+^2$ if $\pi_k\in\mathbb{Z}_+^2$ and $\pi_{k+1}-\pi_k\in\{\alpha_1, -\alpha_2\}$ for each $k\in \Z$. To each edge along a down-right path we associate the random variable
\[
\Lambda_{\{\pi_{k-1},\pi_k\}} :=
\begin{cases}
R^1_{\pi_k} & \text{ if } \{\pi_{k-1},\pi_k\} \text{ is horizontal,}\\
R^2_{\pi_{k-1}} & \text{ if } \{\pi_{k-1},\pi_k\} \text{ is vertical}.
\end{cases}
\]
The following definition is a weaker form of the Burke property (see Theorem 3.3 of \cite{S2012}). 
\begin{definition}
Say the polymer model has the {\rm down-right property} if for all down-right paths $\pi=\{\pi_k\}_{k\in\mathbb{Z}}$, the random variables 
\begin{equation*}
\Lambda(\pi) := \{\Lambda_{\{\pi_{k-1},\pi_k\}}: k\in\mathbb{Z}\}
\end{equation*}
are independent and each $R^1_{\pi_k}$ and $R^2_{\pi_k}$ appearing in the collection are respectively distributed as $R^1$ and $R^2$.
\end{definition}
The partition functions satisfy the recurrence relation
\begin{equation}
Z_x=Y^1_xZ_{x-\alpha_1}+Y^2_x Z_{x-\alpha_2} \qquad \text{for } x\in\mathbb{N}^2.\label{equation Z recursion}
\end{equation}
This recurrence relation then implies the recursions
\begin{align}
\begin{split}
R^1_x &= Y^1_x+Y^2_x\frac{R^1_{x-\alpha_2}}{R^2_{x-\alpha_1}}\\
R^2_x &= Y^1_x\frac{R^2_{x-\alpha_1}}{R^1_{x-\alpha_2}}+Y^2_x
\end{split}
\qquad \text{for } x\in\mathbb{N}^2.
\label{equation ratio recursions}
\end{align}
Using the recursions \eqref{equation ratio recursions} we can reduce the down-right property to a simple preservation in distribution.
\begin{lemma}\label{lemma down-right equiv}
Let $R^1,R^2,(Y^1,Y^2)$ be positive random variables such that $R^1, R^2$ and the pair $(Y^1,Y^2)$ are independent.  Put 
\begin{equation*}
\big(\t{R}^1,\t{R}^2\big):=(Y^1+Y^2 R^1/R^2,\, Y^1 R^1/R^2+Y^2).
\end{equation*}
Then the polymer model with edge weights $R^1, R^2, (Y^1,Y^2)$ has the down-right property if and only if $\big(\t{R}^1,\t{R}^2\big)\overset{d}{=}\big(R^1,R^2\big)$.
\end{lemma}
\begin{proof}[Proof of Lemma \ref{lemma down-right equiv}]
Given a down-right path $\pi$, define its lower-left interior
\[
\text{Int}(\pi):=\{x\in\mathbb{Z}_+^2 \text{ such that } x+(m,n)\in\{\pi\} \text{ for some } m,n\in \N \}.
\]
If the polymer model with edge weights $R^1, R^2, (Y^1,Y^2)$ has the down-right property, taking $\pi$ to be the unique down-right path with interior $\{(0,0)\}$ implies that $(R^1_{1,1}, R^2_{1,1} )\overset{d}{=}(R^1,R^2)$. Then \eqref{equation ratio recursions} and the fact that $\big(R^1_{1,0}, R^2_{0,1}, (Y^1_{1,1},Y^2_{1,1})\big)\overset{d}{=}\big(R^1, R^2, (Y^1,Y^2)\big)$ imply that $(\t{R}^1, \t{R}^2) \overset{d}{=}(R^1,R^2)$.

For the converse direction, we first prove the statement for $\pi$ with finite interior. The case when the interior is empty is true by assumption. Assume that the down-right property holds for all paths $\pi$ with $|\text{Int}(\pi)|= n$. Given a path $\pi$ with $|\text{Int}(\pi)|=n+1$ there exists $x$ such that $\pi$ traverses the right-down corner $\{x-\alpha_1,x,x-\alpha_2\}$. Let $\t{\pi}$ be the path which traverses the same points as $\pi$ with the exception of instead passing through the down-right corner $\{x-\alpha_1, x-\alpha_1-\alpha_2, x-\alpha_2\}$. Then $|\text{Int}(\t{\pi}) | = n$ and so $\big(R^1_{x-\alpha_2}, R^2_{x-\alpha_1}\big)\overset{d}{=}\big(R^1, R^2\big)$. Using \eqref{equation ratio recursions}, the assumption that $(\t{R}^1, \t{R}^2) \overset{d}{=}(R^1,R^2)$ and the independence of $(Y_x^1,Y_x^2)$ from collection $\Lambda(\t{\pi})$ gives us that the collection $\Lambda(\pi)$ has the desired property.

To prove the statement for arbitrary $\pi$, pick a finite subcollection $F$ of $\Lambda(\pi)$. Then there exists $\t{\pi}$ such that $\text{Int}(\t{\pi})$ is finite and $F\subset\Lambda(\t{\pi})$. Since the statement holds for down-right paths with finite interior, we are done.
\end{proof}

\begin{proposition}\label{proposition 4 models have DR property}
	Each of the four basic beta-gamma models, \eqref{model-IG} through \eqref{model-IB}, possesses the down-right property.
\end{proposition}
\begin{proof}
The $(\t{R}^1,\t{R}^2)\overset{d}{=}(R^1,R^2)$ condition in Lemma \ref{lemma down-right equiv} has been checked for the inverse-gamma, gamma, beta, and inverse-beta models by Lemma 3.2 of \cite{S2012}, Lemma 6.3 of \cite{CSS2015}, Lemma 3.1 of \cite{BRS2016}, and Proposition 3.1 of \cite{T2016} respectively. 
\end{proof}

The following lemma is an immediate consequence of the down-right property and the starting point for the proof of Proposition \ref{proposition variance formula}.

\begin{lemma}\label{lemma DR consequences}
If the polymer model with edge weights $R^1, R^2$, $(Y^1,Y^2)$ possesses the down-right property, then for all $(m,n)\in \Z^2_{+}$, 
\begin{enumerate}
\item $\E[\log Z_{m,n}] = m\E[\log R^1] + n\E[\log R^2]$,
\item $\Var[\log Z_{m,n}]= -m\Var[\log R^1]+n\Var[\log R^2]+2\Cov(S_N,S_S)$, 
\item $\Var[\log Z_{m,n}]= m\Var[\log R^1]-n\Var[\log R^2]+2\Cov(S_E,S_W)$,
\end{enumerate}
where
\begin{equation}\label{equation S's}
\begin{aligned}
 S_N&:=\log{Z_{m,n}}-\log{Z_{0,n}}=\sum_{i=1}^m\log{R^1_{i,n}}, &
 S_S&:=\log{Z_{m,0}}=\sum_{i=1}^m\log{R^1_{i,0}},\\
 S_E&:=\log{Z_{m,n}}-\log{Z_{m,0}}=\sum_{j=1}^n\log{R^2_{m,j}}, &
 S_W&:=\log{Z_{0,n}}=\sum_{j=1}^n\log{R^2_{0,j}}.
\end{aligned}
\end{equation}
\end{lemma}
\begin{proof}
 
 By the down-right property $S_S$ is independent of $S_W$, $S_N$ is independent of $S_E$, and 
 \begin{align*}
 \mathbb{V}\text{ar}[S_N]&=\mathbb{V}\text{ar}[S_S]=m\Var[\log R^1] &
 \mathbb{V}\text{ar}[S_E]&=\mathbb{V}\text{ar}[S_W]=n\Var[\log R^2].
 \end{align*}
 These facts along with the equalities $\log{Z_{m,n}}=S_N+S_W=S_E+S_S$ gives (a) and  
 \begin{align*}
 \Var[\log{Z_{m,n}}] &= \Var[S_N]+\Var[S_W]+2\Cov(S_N,S_W)\\
  &= \Var[S_N]+\Var[S_W]+2\Cov(S_N,S_E+S_S-S_N)\\
   &= -\Var[S_N]+\Var[S_W]+2\Cov(S_N,S_S)\\
   &= -m\Var[\log R^1]+n\Var[\log R^2]+2\Cov(S_N,S_S).
 \end{align*}
Similarly, 
  \begin{align*}
\Var[\log{Z_{m,n}}] =-n\Var[\log R^2]+m\Var[\log R^1]+2\Cov(S_E,S_W).
 \end{align*} 
\end{proof}

\section{The Mellin transform framework}\label{section Mellin}
In this section we develop a framework which allows us to treat the four basic beta-gamma models simultaneously.

Given a function $f:(0,\infty)\rightarrow[0,\infty)$, write $M_f$ for its Mellin transform
 \[
M_f(a):= \int_0^\infty x^{a-1}f(x)dx
\]
for any $a\in\R$ such that the integral converges. Define 
\[
D(M_f):=\text{interior}(\{a\in\R : 0<M_f(a)<\infty\}).
\]
\begin{definition}
Given a function $f:(0,\infty)\to [0,\infty)$ such that $D(M_f)$ is non-empty, we define a family of densities on $(0,\infty)$ parametrized by $a\in D(M_f)$:
\begin{equation}\label{equation Mellin density}
\rho_{f,a}(x):= M_f(a)^{-1}x^{a-1}f(x).
\end{equation}
We write $X\sim m_f(a)$ to denote that the random variable $X$ has density $\rho_{f,a}$.
\end{definition}

\begin{remark}\label{remark - Mellin consequences}
If $f:(0,\infty)\to [0,\infty)$ is such that $D(M_f)$ is non-empty, then $M_f$ is $C^\infty$ throughout $D(M_f)$. Furthermore, if $X\sim m_f(a)$, then 
\begin{enumerate}
\item $\log X$ has finite exponential moments. That is, there exists some $\epsilon>0$ such that 
\[
\E[e^{\epsilon |\log X|}]\leq \E[X^{\epsilon}]+\E[X^{-\epsilon}]=\frac{M_f(a+\epsilon)+M_f(a-\epsilon)}{M_f(a)}<\infty.
\]
\item For all $k\in \N$,
\[\frac{\p^k}{\p a^k} M_f(a)=M_f(a) \E[(\log X)^k].\]
\item $\E [\log X]= \psi_0^f(a)$ and $\Var[\log X] = \psi_1^f(a)$, where 
\[
\psi^f_n(a):=\frac{\p^{n+1}}{\p a^{n+1}}\log{M_f(a}) \text{ for } n\in\mathbb{Z}_+.
\]
\end{enumerate}
\end{remark}

The following remark says that random variables with densities of the form \eqref{equation Mellin density} are closed under inversion.

\begin{remark}\label{remark - Mellin inversion}
If $f:(0,\infty)\rightarrow [0,\infty)$ is such that $D(M_f)$ is non-empty and $g(x):=f(\frac{1}{x})$ for $x\in (0,\infty)$, then for all $a\in D(M_f)$,
\begin{enumerate}
\item  $X\sim m_f(a) \Leftrightarrow X^{-1}\sim m_g(-a)$,
\item $M_f(a)=M_g(-a)$  and therefore  $D(M_g)=-D(M_f)$,
\item  $\psi_n^f(a)=(-1)^{n+1} \psi_g^f(-a)$ for all $n\in \N$.
\end{enumerate}
\end{remark}

\begin{definition}
Let $f^j:(0,\infty) \to [0,\infty)$ be such that $D(M_{f^j})$ is non-empty for $j=1,2$. We say that the polymer environment is {\rm Mellin-type with respect to $(f^1, f^2)$} if $(R^1, R^2)\sim m_{f^1}(a_1)\otimes m_{f^2}(a_2)$ for some $a_j \in D(M_{f^j})$.
\end{definition}
When the polymer environment is Mellin-type with parameters $(a_1, a_2)$, we use $\P^{(a_1,a_2)}$, $\E^{(a_1,a_2)}$, $\Var^{(a_1,a_2)}$, $\Cov^{(a_1,a_2)}$ in place of $\P$, $\E$, $\Var$, $\Cov$ respectively.

\subsection{The four basic beta-gamma models are Mellin-type}

We first specify functions $f$ to obtain each of the random variables appearing in the four basic beta-gamma models. Note that the fourth column in Table \ref{table f} specifies the distribution of the random variable corresponding to $f$. We let $\text{B}(a,b) = \frac{\Gamma(a) \Gamma(b)}{\Gamma(a+b)}$ denote the beta function and $\Psi_n(x)= \frac{\p^{n+1}}{\p x^{n+1}}\log \Gamma(x)$ the polygamma function of order $n$. For the Table \ref{table f} we assume $b>0$ and $a\in D(M_f)$.
\begin{figure}[H]
\begin{equation}
	\begin{array}{|c|c|c|c|c|}
		\hline  f(x) & D(M_f) & M_f(a) & m_f(a) & \psi_n^f(a)   \\
		\hline \hline
		 e^{-bx} & (0,\infty) & \Gamma(a)/b^a & \text{Ga}(a,b) & \Psi_n(a) - \delta_{n,0}\log b  \\ \hline
		 e^{-b/x} & (-\infty,0) & \Gamma(-a) b^a &\text{Ga}^{-1}(-a,b) & (-1)^{n+1}(\Psi_n(-a)-\delta_{n,0}\log b)   \\ \hline
		 (1-x)^{b-1}\ind_{\{0<x<1\}} & (0,\infty) & \text{B}(a,b) &\text{Be}(a,b) & \Psi_n(a)-\Psi_n(a+b)   \\ \hline
		 (1-\frac{1}{x})^{b-1}\ind_{\{x>1\}} & (-\infty,0) & \text{B}(-a,b) &\text{Be}^{-1}(-a,b) & (-1)^{n+1}(\Psi_n(-a)-\Psi_n(-a+b))   \\ \hline
		 (\frac{x}{x+1})^{b} & (-b,0) & \text{B}(-a,b+a) & \text{Be}^{-1}(-a,b+a)-1 & \Psi_n(a+b) + (-1)^{n+1}\Psi_n(-a)   \\ \hline
	\end{array}
\end{equation}
\caption{Mellin framework data for the distributions appearing in the four basic beta-gamma models.}
\label{table f}
\end{figure}



To express the distribution of the polymer environment in each of the four basic beta-gamma models given in \eqref{model-IG} through \eqref{model-IB} within this Mellin framework, we let 
\begin{equation}\label{polymer environment distribution}
	(R^1,R^2,X) \sim m_{f^1}(a_1) \otimes m_{f^2}(a_2) \otimes m_{f^1}(a_3),
\end{equation}
where the functions $f^1$, $f^2$ and parameters $a_j$, $j=1,2,3$ are given in Table \ref{table 4 models}. Recall that in each of the models, $(Y^1,Y^2)$ are given in terms of $X$. For Table \ref{table 4 models} we assume $\mu,\beta>0$.
\begin{figure}[H] 
\[
	\begin{array}{|l||c|c|c|l|}
		\hline
		\text{Model} & f^1(x) & f^2(x) & (a_1,a_2,a_3)& \\ \hline \hline
		\text{IG} & e^{-\beta/x} & e^{-\beta/x} & (\theta-\mu,-\theta,-\mu)& \theta\in(0,\mu) \\ \hline 
		\text{G} & e^{-\beta x} & (1-\frac{1}{x})^{\mu-1}\ind_{\{x>1\}} & (\mu +\theta,-\theta,\mu) & \theta\in(0,\infty)\\ \hline
		\text{B} & (1-x)^{\beta-1}\ind_{\{0<x<1\}} & (1-\frac{1}{x})^{\mu-1}\ind_{\{x>1\}} & (\mu +\theta,-\theta,\mu)& \theta\in (0,\infty) \\ \hline
		\text{IB} &  (1-\frac{1}{x})^{\beta-1}\ind_{\{x>1\}} & (\frac{x}{x+1})^{(\beta+\mu)} & (\theta-\mu,-\theta,-\mu)&\theta\in(0,\mu) \\ \hline 
	\end{array}
\]
\caption{Functions and parameters to fit the four basic beta-gamma models into the Mellin framework.}
\label{table 4 models}
\end{figure}
\begin{remark}\label{remark parameter DR property}
For each fixed value of the bulk parameter $a_3$, we obtain a family of models with boundary parameters $a_1$ and $a_2$ satisfying $a_1 + a_2 = a_3$. For any such $a_1$ and $a_2$, by Proposition \ref{proposition 4 models have DR property} these models will have the down-right property.
\end{remark}

\subsection{Coupling of polymer environments}
In order to compare polymer environments with different parameters, we use a coupling to express the boundary weights as functions of i.i.d.\ uniform$(0,1)$ random variables.

If $f:(0,\infty) \to [0,\infty)$ is such that $D(M_f)$ is non-empty, write $F^f$ for the CDF of the random variable $X\sim m_f(a)$.  Specifically, $F^f:D(M_f)\times [0,\infty)\rightarrow [0,1]$ is given by
\[
F^f(a,x)=\frac{1}{M_f(a)}\int_0^xy^{a-1}f(y)dy.
\]
Define the quantile function 
\begin{equation} 
H^f(a,p) := \inf\{x: p\leq F^f(a,x)\}.\label{eq def Hf}
\end{equation}
If the random variable $\eta$ is uniformly distributed on the interval $(0,1)$, then $H^f(a,\eta)\sim m_f(a)$. 

Suppose that a polymer environment $\omega$ is Mellin-type with respect to $(f^1,f^2)$ with parameters $(b_1,b_2)$. Let $\{\eta_i^1,\eta_j^2:i,j\in\N\}$ be i.i.d.\ uniform$(0,1)$ random variables that are independent of the bulk weights $\{(Y^1_z,Y^2_z):z\in \N^2\}$.  Write $\hat{\P}$, $\widehat{\E}$, and $\hat{\Var}$ for the probability measure and the corresponding expectation and variance of these uniform random variables and the bulk weights. Define the coupled environment
\begin{equation}
\omega^{(b_1,b_2)}:=\{H^{f^1}(b_1,\eta^1_i),H^{f^2}(b_2,\eta^2_j), (Y^1_z, Y^2_z): i\in\N,j\in\N, z\in \N^2\}.\label{def coupled environment}
\end{equation}
Note that this environment is equal in distribution to the original environment $\omega$.

To specifically denote weights accumulated by a path,  the partition function, the quenched measure, and the annealed expectation, associated to the coupled environment $\omega^{(b_1,b_2)}$, define
\begin{equation}
\begin{split}
W(b_1,b_2)(x_\centerdot) &:=\prod_{k=1}^{m+n}\omega^{(b_1,b_2)}_{(x_{k-1},x_k)}  \qquad \text{ for } x_\centerdot\in\Pi_{m,n}\\
Z_{m,n}(b_1,b_2) &:=\sum_{x_\centerdot\in\Pi_{m,n}}W(b_1,b_2)(x_\centerdot)\\
Q_{m,n}^{(b_1,b_2)}(A) &:= \frac{1}{Z_{m,n}(b_1,b_2)}\sum_{x_\centerdot\in A} W(b_1,b_2)(x_\centerdot) \qquad \text{ for } A\subset \Pi_{m,n}\\
E_{m,n}^{(b_1,b_2)} [\bullet] &:= \hat{\E}\left[E^{Q_{m,n}^{(b_1,b_2)}}[\bullet]\right].
\end{split}\label{equations coupling} 
\end{equation}
Recall the definition of the exit points $t_j$ \eqref{def exit points}. We can decompose the weight accumulated along a path to isolate the dependence on boundary weights 
\begin{equation}
W(b_1,b_2)(x_\centerdot)=\prod_{i=1}^{t_1} H^{f^1}(b_1,\eta^1_i)\prod_{j=1}^{t_2}H^{f^2}(b_2,\eta_j^2)\prod_{k=(t_1\vee t_2)+1}^{m+n}\omega^{(b_1,b_2)}_{(x_{k-1},x_k)}. \label{equation coupled weight decomp}
\end{equation}
Notice that one of the first two products will be empty and the third product involves only the bulk weights.

If we assume that $f:(0,\infty)\to [0,\infty)$ has open support, is continuous on its support, and $D(M_f)$ is non-empty, then $F^f$ is continuously differentiable on the set $D(M_f)\times \supp(f)$. By the implicit function theorem, $H^f$ is continuously differentiable and
for all $(a,p)\in D(M_f)\times (0,1)$, we have 
\begin{equation}
\frac{\partial H^f}{\partial a}(a,p)=\frac{-\frac{\partial F^f}{\partial a}(a,H^f(a,p))}{\frac{\partial F^f}{\partial x}(a,H^f(a,p))}=H^f(a,p)L^f(a,H^f(a,p))\label{equation partial H}
\end{equation}
where $L^f$ is given by
\begin{equation}
L^f(a,x):=\frac{x^{-a}}{f(x)}\int_0^x(\psi^f_0(a)-\log y)y^{a-1}f(y)dy=-\frac{x^{-a}}{f(x)}\int_x^{\infty}(\psi^f_0(a)-\log y)y^{a-1}f(y)dy.\label{eq Lf integral form}
\end{equation}
The second equality follows from the definition of $\psi_0^f(a)$.
Notice that \[L^f(a,x) = -\frac{1}{x\rho_{f,a}(x)} \Cov(\log X, \ind_{\{X\leq x\}})= L_X(x) \text{ (as defined in \eqref{definition L_X})}\]
when $X\sim m_f(a)$, and therefore $L^f(a,x)\geq 0$.

The following hypothesis collects technical conditions for the function $f$ used in the sequel.

\begin{hypothesis}\label{hypothesis CSS}
Suppose that $f:(0,\infty) \to [0,\infty)$ is such that $D(M_f)$ is non-empty, $f$ has open support, is differentiable on its support, and for all compact $K\subset D(M_f)$ there exists a constant $C$ depending only on $K$ such that the following hold for all $a \in K$:
\begin{gather}
L^f(a,x)\leq C(1+|\log x|)\qquad \text{for all } x\in \spt(f),\label{CSS 1}\\
\int_0^1\Bigl|\frac{\p}{\p a}L^f(a,H^f(a,p))\Bigr|dp\leq C.\label{CSS 2}
\end{gather}  
\end{hypothesis}

\begin{remark}\label{remark f is CSS}
If $X\sim m_f(a)$ where $f$ satisfies Hypothesis \ref{hypothesis CSS}, then by \eqref{CSS 1} and Remark \ref{remark - Mellin consequences}, $L_X(X)$ has finite exponential moments. By Lemma \ref{lemma f are CSS} in the appendix, each of the functions $f$ corresponding to the random variables appearing in the four basic beta-gamma models (see Table \ref{table f}) satisfies Hypothesis \ref{hypothesis CSS}.
\end{remark}

\begin{lemma}\label{lemma covariance explicit}
Assume that the polymer environment is Mellin-type with respect to $(f^1,f^2)$, where $f^1$ and $f^2$ satisfy Hypothesis \ref{hypothesis CSS}. Further assume that $\log Y^1$ and $\log Y^2$ have finite variance. Recall the notation \eqref{equation S's}. Then for all $(m,n)\in \Z_+^2$ 
\begin{align}
\Cov(S_N, S_S) &= E_{m,n}[\sum_{i=1}^{t_1} L_{R^1}(R^1_{i,0})],\label{equation cov explicit 1}\\
\Cov(S_E, S_W) &= E_{m,n}[\sum_{j=1}^{t_2} L_{R^2}(R^2_{0,j})].\label{equation cov explicit 2}
\end{align}
\end{lemma}

\begin{proof}
By assumption, there exists $(a_1,a_2) \in D(M_{f^1}) \times D(M_{f^2})$ such that $(R^1,R^2)\sim m_{f^1}(a_1) \otimes m_{f^2}(a_2)$. There exist open neighborhoods $U_j$ about $a_j$ contained in $D(M_{f^j})$ for $j=1,2$. We then show that 
\begin{align}
\frac{\p}{\p b_1}\E^{(b_1,a_2)}[S_N]&=\Cov^{(b_1,a_2)}(S_N,S_S)  \text{ for all } b_1\in U_1 \label{equation cov deriv 1},\\
\frac{\p}{\p b_2}\E^{(a_1,b_2)}[S_E]&=\Cov^{(a_1,b_2)}(S_E,S_W)  \text{ for all } b_2\in U_2, \label{equation cov deriv 2}
\end{align}
and that the mappings $b_1\mapsto \Cov^{(b_1,a_2)}(S_N,S_S)$ and $b_2\mapsto \Cov^{(a_1,b_2)}(S_E,S_W)$ are continuous.
We begin with \eqref{equation cov deriv 1}. We will vary the parameter $b_1$ of the weights $R^1_{i,0}$ while keeping the parameter $a_2$ of the weights $R^2_{0,j}$ fixed. Let $\widetilde{\E}$ be the expectation over $\{R^2_{0,j},(Y^1_{x},Y^2_x)\}_{j\in\N,x\in\N\times\N}$. By Remark \ref{remark - Mellin consequences} and Lemma \ref{lemma finite square moments}, $\E^{(b_1, a_2)}[S_N^2]<\infty$ for all $b_1\in U_1$. Then $\E^{(b_1,a_2)}[S_N]=\E^{b_1}[\widetilde{\E}[S_N]]$ where $\E^{b_1}$ denotes the expectation over $\{R^1_{i,0}\}_{i=1}^m$ when $R^1\sim m_{f^1}(b_1)$. We now invoke Lemma \ref{lem cov calc}. Specifically, we use $r=m$, $X_k=R^1_{k,0}$, $f_k=f^1$ for all $k=1,\dots,m$ and $A(R^1_{1,0}, \ldots, R^1_{m,0})=\widetilde{\E}[S_N]$ to get, for all $b_1\in U_1$,
\begin{align*}
\frac{\p}{\p b_1}\E^{(b_1,a_2)}[S_N]&= \frac{\p}{\p b_1} \E^{b_1}[A(X_1,\ldots, X_m)]=\Cov^{b_1}(A(X_1,\ldots, X_m),S_S)= \Cov^{(b_1,a_2)}(S_N,S_S)
\end{align*}
 and $U_1\owns b_1\mapsto \Cov^{(b_1,a_2)}(S_N,S_S)$ is continuous. The third equality follows from the fact that the collection $\{R^2_{0,j},(Y^1_x,Y^2_x)\}_{j\in\N,x \in\N\times\N}$ is independent of $S_S$.   The second moment condition of Lemma \ref{lem cov calc} is satisfied since 
\[
\E^{b_1}[A(X_1,\ldots, X_r)^2]=\E^{b_1}[(\widetilde{\E}[S_N])^2]\leq \E^{b_1}[\widetilde{\E}[S_N^2]]=\E^{(b_1,a_2)}[S_N^2]<\infty \text{ for all } b_1\in U_1.
\] 
A similar argument yields \eqref{equation cov deriv 2}.

Using the coupling \eqref{def coupled environment}
\begin{align}\label{equation coupled L sum}
E_{m,n}[\sum_{i=1}^{t_1} L_{R^1}(R^1_{i,0})] = E_{m,n}^{(a_1,a_2)} [\sum_{i=1}^{t_1}L^{f^1}(a_1,H^{f^1}(a_1, \eta_i^1))].
\end{align}
Taking the derivative of \eqref{equation coupled weight decomp} and using \eqref{equation partial H}, for $j=1,2$
\begin{equation}
\frac{\p}{\p b_j}\log(W(b_1,b_2)(x_\centerdot))=\sum_{k=1}^{t_j}\frac{\p}{\p b_j}\log H^{f^j}(b_j,\eta_k^j)=\sum_{k=1}^{t_j}L^{f^j}(b_j,H^{f^j}(b_j\eta_k^j)). \label{cov calc 2}
\end{equation}
Therefore
\begin{equation}
\frac{\p}{\p b_j}W(b_1,b_2)(x_\centerdot)=W(b_1,b_2)(x_\centerdot)\sum_{k=1}^{t_j}L^{f^j}(b_j,H^{f^j}(b_j\eta_k^j))\label{equation deriv W}
\end{equation}
which implies that 
\begin{equation}
\frac{\p}{\p b_j}\log Z_{m,n}(b_1,b_2)=E^{Q_{m,n}^{(b_1,b_2)}}[\sum_{k=1}^{t_j}L^{f^j}(b_j,H^{f^j}(b_j,\eta_k^j))].\label{equation deriv logZ}
\end{equation}
We now prove \eqref{equation cov explicit 1}. Similar to \eqref{equations coupling}, in the coupled environment we use $S_\bullet(b_1,b_2)$ to make explicit the dependence of $S_\bullet$ on the parameters $b_1$ and $b_2$. Recall that $\hat{\E}$ is the expectation of the coupled environment. For $\epsilon>0$ small enough such that $[a_1-\epsilon,a_1+\epsilon]\subset U_1$,
\begin{equation}
\begin{split}
\int_{a_1-\epsilon}^{a_1+\epsilon} \Cov^{(b_1,a_2)}(S_N,S_S) db_1&=\E^{(a_1+\epsilon,a_2)}[S_N]-\E^{(a_1-\epsilon,a_2)}[S_N]\\
&=\hat{\E}[S_N(a_1+\epsilon,a_2)-S_N(a_1-\epsilon,a_2)]\\
&=\hat{\E}[\int_{a_1-\epsilon}^{a_1+\epsilon}\frac{\p}{\p b_1}\log Z_{m,n}(b_1,a_2)db_1]\\
&=\int_{a_1-\epsilon}^{a_1+\epsilon}\hat{\E}[\frac{\p}{\p b_1}\log Z_{m,n}(b_1,a_2)]db_1
\end{split} \label{equation integral equality}
\end{equation}
where the first equality follows from \eqref{equation cov deriv 1}, the third equality follows because $S_W$ does not depend on $b_1$ and $S_N(b_1,a_2)=\log Z_{m,n}(b_1,a_2)-S_W(a_2)$. The last equality follows from \eqref{equation deriv logZ} and Tonelli's theorem (by the non-negativity of $L^{f^1}$).

Recall that $b_1\mapsto \Cov^{(b_1,a_2)}(S_N,S_S)$ is continuous. Once we show that the mapping 
\begin{equation}\label{equation cov mapping}
b_1\mapsto \hat{\E}[\frac{\p}{\p b_1}\log Z_{m,n}(b_1,a_2)] = E_{m,n}^{(b_1,a_2)}[\sum_{i=1}^{t_1}L^{f^1}(b_1,H^{f^1}(b_1, \eta^1_i))]
\end{equation}
is continuous, using \eqref{equation integral equality} and \eqref{equation coupled L sum} we will have \eqref{equation cov explicit 1}. The continuity of \eqref{equation cov mapping} follows from the continuity of $b_1\mapsto E^{Q_{m,n}^{(b_1,a_2)}} [\sum_{k=1}^{t_1} L^{f^1}(b_1,H^{f^1}(b_1,\eta^1_k))]$, the dominated convergence theorem, and the bound
\begin{align*}
&\ \hat{\E} \big[\sup_{|b_1-a_1|\leq\epsilon} E^{Q_{m,n}^{(b_1,a_2)}} [\sum_{k=1}^{t_1} L^{f^1}(b_1,H^{f^1}(b_1,\eta^1_k))]\big]\\
&\leq \hat{\E}[\sup_{|b_1-a_1|\leq\epsilon}\sum_{k=1}^{m} L^{f^1}(b_1,H^{f^1}(b_1,\eta^1_k))]\\
&\leq C\hat{\E}[\sum_{k=1}^m 1+ |\log H^{f^1}(a_1-\epsilon, \eta^1_k)| + |\log H^{f^1}(a_1 + \epsilon, \eta^1_k)|]<\infty
\end{align*}
where we use the non-negativity of $L^{f^1}$ to replace $t_1$ by its upper bound $m$, then use assumption \eqref{CSS 1} of Hypothesis \ref{hypothesis CSS} (with the fact that $H^{f^1}(b,x)$ is non-decreasing in $b$) and part (a) of Remark \ref{remark - Mellin consequences}.

A similar argument shows that 
\[
\Cov^{(a_1,a_2)}(S_E,S_W)=E_{m,n}^{(a_1,a_2)}[\sum_{j=1}^{t_2}L^{f^2}(a_2,R^2_{0,j})].
\]
This completes the proof.
\end{proof}
 
We can now give the proof of Proposition \ref{proposition variance formula}.

\begin{proof}[Proof of Proposition \ref{proposition variance formula}]
By assumption, the polymer environment is distributed as in \eqref{polymer environment distribution}, where $f^1$ and $f^2$ satisfy Hypothesis \ref{hypothesis CSS} by Remark \ref{remark f is CSS}. By Remark \ref{remark - Mellin consequences}, for each of the four models $\log u$ and $\log v$ have finite variance. Thus the conditions of Lemma \ref{lemma covariance explicit} are satisfied. Combining Proposition \ref{proposition 4 models have DR property} with Lemma \ref{lemma DR consequences}, and Lemma \ref{lemma covariance explicit} yields the result.
\end{proof}

\section{Proof of variance upper bound}\label{section variance UB}
The first lemma of this section allows us to compare the variance of the free energy at different parameter values.
\begin{lemma}\label{lemma variance comparison}
Assume that the polymer environment is distributed as in \eqref{polymer environment distribution}. Let $\epsilon$ be small enough such that for all $|\lambda|\leq\epsilon$, $a_1+\lambda\in D(M_{f^1})$ and $a_2-\lambda\in D(M_{f^2})$. Then there exists a positive constant $C$ depending only on $(a_1,a_2),\, \beta,$ and $\epsilon$ such that for all $(m,n)\in\N^2$, the following holds for all $|\lambda|\leq\epsilon$,
\begin{align*}
	\left|\Var^{(a_1,a_2)}[\log Z_{m,n}]- \Var^{(a_1+\lambda,a_2-\lambda)}[\log Z_{m,n}] \right| \leq C(m+n)|\lambda|
\end{align*}
\end{lemma}
\begin{proof}
Let $\t{a}_1 = a_1 + \lambda$ and $\t{a}_2 = a_2 - \lambda$. Applying Proposition \ref{proposition variance formula} (recalling that $\psi_1^{f}(a) = \Var[\log X]$ when $X\sim m_f(a)$) then using the coupling \eqref{equations coupling} yields, for $j=1,2$:
  \begin{align}
    &\frac{1}{2} \left( \Var^{(\ta_1,\ta_2)}[\log Z_{m,n}]-\Var^{(a_1,a_2)}[\log Z_{m,n}] \right) \label{vl-Vdiff}\\
    &= \frac{(-1)^j}{2} \left[m\big(\psi_1^{f^1}(\ta_1) - \psi_1^{f^1}(a_1)\big) -n
    \big(\psi_1^{f^2}(\ta_2) - \psi_1^{f^2}(a_2)\big)\right] \label{vl-psi-diff}\\
    &\quad  + E_{m,n}^{(\ta_1,\ta_2)}\left[\sum_{k=1}^{t_j}L^{f^j}(\ta_j,H^{f^j}(\ta_j, \eta^j_k))\right] - E_{m,n}^{(a_1,
      a_2)}\left[\sum_{k=1}^{t_j}L^{f^j}(a_j,H^{f^j}(a_j, \eta^j_k))\right].\label{vl-L-diff}
  \end{align}
  Since $\psi_1^{f^1}$ and $\psi_1^{f^2}$ are continuously
  differentiable, there is a constant $C_1$ such that line
  \eqref{vl-psi-diff} is bounded by $C_1(m+n)|\lambda|$.
  Suppressing the $m,n$ dependence, we then split line
  \eqref{vl-L-diff} as
  \begin{align}
    &=\hat{\E}
    E^{Q^{(\ta_1, \ta_2)}}\left[\sum_{k=1}^{t_j}L^{f^j}(\ta_j,H^{f^j}(\ta_j,\eta_k^j))\right]
    - \hat{\E}
    E^{Q^{(\ta_1,\ta_2)}}\left[\sum_{k=1}^{t_j}L^{f^j}(a_j,H^{f^j}(a_j,\eta_k^j))\right]
    \label{vl-L-diff1}\\
    &\quad + \hat{\E}
    E^{Q^{(\ta_1,\ta_2)}}\left[\sum_{k=1}^{t_j}L^{f^j}(a_j,H^{f^j}(a_j,\eta_k^j))\right]
    -
    \hat{\E} E^{Q^{(a_1,a_2)}}\left[\sum_{k=1}^{t_j}L^{f^j}(a_j,H^{f^j}(a_j,\eta_k^j))\right].\label{vl-L-diff2}
\end{align}
For line \eqref{vl-L-diff1}, since $t_j$ is all that is random under $E^{Q^{(\ta_1, \ta_2)}}$, we can replace $t_j$ by $m\vee n$. Thus
\begin{align}
|\, \text{line }\eqref{vl-L-diff1}| &\leq \Eh\sum_{k=1}^{m\vee n} \left| L^{f^j}(\ta_j,H^{f^j}(\ta_j,\eta^j_k)) - L^{f^j}(a_j,H^{f^j}(a_j,\eta^j_k))\right|\nonumber \\
	&= (m\vee n) \int_0^1 \left| L^{f^j}(\ta_j,H^{f^j}(\ta_j,\eta)) - L^{f^j}(a_j,H^{f^j}(a_j,\eta))\right| d\eta \nonumber \\
    &= (m\vee n) \int_0^1 \left| \int^{\ta_j}_{a_j} \frac{\p }{\p a}L^{f^j}(a,H^{f^j}(a,\eta))  da \right|d\eta \nonumber \\
    &\leq  (m\vee n) \left|\int^{\ta_j}_{a_j} \int_0^1 \left|\frac{\p }{\p a}L^{f^j}(a,H^{f^j}(a,\eta)) \right| d\eta da\right| \nonumber \\
    &\leq (m\vee n) C_2 |\lambda|.\label{vl-est1}
\end{align}
  In the last step we used the fact that $f^j$ satisfy assumption \eqref{CSS 2} in Hypothesis \ref{hypothesis CSS} by Remark \ref{remark f is CSS}. 

We can write line \eqref{vl-L-diff2} as
\begin{align*}
\hat{\E}\big[\sum_{k=1}^{\ell_j} L^{f^j}(a_j,H^{f^j}(a_j,\eta^j_k))\big(Q^{(\ta_1, \ta_2)}(t_j\geq k)-Q^{(a_1, a_2)}(t_j\geq k) \big)\big],
\end{align*}
where $\ell_1 = m$ and $\ell_2 = n$. By Lemma \ref{lemma Q(t_j>k)}, $Q^{(a_1+\lambda, a_2-\lambda)}(t_1\geq k)$ is stochastically non-decreasing in $\lambda$, and $Q^{(a_1+\lambda, a_2-\lambda)}(t_2\geq k)$ is stochastically non-increasing in $\lambda$. Using the bound on \eqref{vl-psi-diff}, the bound \eqref{vl-est1}, and the non-negativity of $L^{f^j}$, line \eqref{vl-L-diff2} is non-negative if $j=1$ and $\lambda>0$ or $j=2$ and $\lambda<0$. This implies 
\[\eqref{vl-Vdiff} \geq -C(m+n)|\lambda|.\]
If $j=2$ and $\lambda>0$ or $j=1$ and $\lambda<0$, then \eqref{vl-L-diff2} is non-positive, so
\[\eqref{vl-Vdiff} \leq C(m+n)|\lambda|.\]
This completes the proof.
\end{proof}

\begin{lemma}\label{E[t] lemma}
Assume that the polymer environment is distributed as in \eqref{polymer environment distribution}. Then there exists a positive constant $C$ depending on $(a_1,a_2)$ and $\beta$ such that for all $(m,n)\in\Z_+^2$ the following two inequalities hold:
\begin{align*}
E_{m,n}[\sum_{i=1}^{t_1}L_{R^1}(R^1_{i,0})]&\leq C(E_{m,n}[t_1]+1), &
E_{m,n}[\sum_{j=1}^{t_2}L_{R^2}(R^2_{0,j})]&\leq C(E_{m,n}[t_2]+1).
\end{align*}
\end{lemma}
\begin{proof}
Let $L_i=L_{R^1}(R^1_{i,0})$, $\overline{L}_i =L_i-\E[L_i]$, and $S_k= \sum_{i=1}^k \overline{L}_i$. Note that $L_i\sim L_{R^1}(R^1)$ has finite exponential moments by Remark \ref{remark f is CSS}. We first estimate
\begin{align*}
\E\left[\ind_{\{S_k>k\}}S_k\right] \leq \left( \mbb{P}\{S_k>k\}\right)^{1/2}(k\Var L_1)^{1/2} \leq \left( \frac{\E[S_k^8]}{k^8}\right)^{1/2}(kC)^{1/2} \leq Ck^{-3/2}.
\end{align*}
Thus \[\sum_{k=1}^\infty \E\left[\ind_{\{S_k>k\}}S_k\right]\leq C.\] Using this, we then get
\begin{align*}
E_{m,n}\left[\sum_{i=1}^{t_1} L_{R^1}(R^1_{i,0}) \right] &= E_{m,n}\left[\sum_{i=1}^{t_1} \overline{L}_i + \E L_i \right]\\
&= E_{m,n}[t_1] \E[L_1] + E_{m,n}\left[\sum_{i=1}^{t_1} \overline{L}_i\right] \\
&=E_{m,n}[t_1] \E[L_1] + \sum_{k=1}^m \E\left[ Q_{m,n}(t_1=k) S_k \right]\\
&=E_{m,n}[t_1] \E[L_1] + \sum_{k=1}^m \E\left[ \ind_{\{S_k\leq k\}}Q_{m,n}(t_1=k) S_k + \ind_{\{S_k>k\}} Q_{m,n}(t_1=k)S_k\right]\\
&\leq E_{m,n}[t_1] \E[L_1] + \sum_{k=1}^m k\E\left[  Q_{m,n}(t_1=k)\right] +\sum_{k=1}^m \E\left[ \ind_{\{S_k>k\}} S_k \right]\\
&\leq E_{m,n}[t_1] \E[L_1] + E_{m,n}[t_1] + C\\
&\leq C\left(E_{m,n}[t_1]+1\right).
\end{align*}
The proof for $t_2$ is analogous. 
\end{proof}

\begin{proposition}\label{proposition stationary ub}
Assume that the polymer environment is distributed as in \eqref{polymer environment distribution}. Assume that the sequence $(m,n)=(m_N,n_N)_{N=1}^{\infty}$ satisfies 
\[
|m-N\psi_1^{f^2}(a_2)|\vee|n-N\psi_{1}^{f^1}(a_1)|\leq \kappa_N
\]
where $\kappa_N\leq \gamma N^{2/3}$ and $\gamma$ is some positive constant.

Then there exist positive finite constants $C_1,C_2,C_3, \delta,\delta_1$ depending only on $(a_1,a_2),\, \beta,$ and $\gamma$ such that for $N\in\N$ and $ 1\vee C_1 \kappa_N\leq u\leq \delta N$,
\begin{align}
\mbb{P}\big\{Q_{m,n}(t_j\geq u)\geq  e^{\frac{-\delta u^2}{N}}\big\}\leq C_2 \Big(\frac{N^2}{u^4}E_{m,n}[t_j]+\frac{N^2}{u^3}\Big)\nonumber
\end{align}
while for $ N\in\N$ and $1\vee C_1 \kappa_N\vee\delta N\leq u$,
\begin{align}
\mbb{P}\big\{Q_{m,n}(t_j\geq u)\geq e^{-\delta_1 u}\big\}\leq 2 e^{-C_3u}\nonumber
\end{align}
hold simultaneously for both $j=1,2$.
\end{proposition}
\begin{proof}
Let $\epsilon>0$ be small enough such that for all $|\lambda|\leq \epsilon$, $a_1(\lambda):=a_1+\lambda\in D(M_{f^1})$ and $a_2(\lambda):=a_2-\lambda\in D(M_{f^2})$.  For the moment fix $\lambda_1\in [0,\epsilon]$, $\lambda_2\in [-\epsilon,0]$, and $u\geq 1$.  The $\lambda_j$ will give the perturbation $(a_1(\lambda_j),a_2(\lambda_j))$ of the parameters $(a_1,a_2)$ which will be used when dealing with the exit time $t_j$.    Using the coupling in \eqref{equations coupling}, \eqref{equation coupled weight decomp} gives: for both $j=1,2$ and any path $x_{\centerdot}$ such that $t_j(x_{\centerdot})\geq u$,
\begin{align*}
\frac{W(a_1,a_2)(x_{\centerdot})}{W(a_1(\lambda_j),a_2(\lambda_j))(x_{\centerdot})}=\prod_{k=1}^{t_j}\frac{H^{f^j}(a_j,\eta_k^j)}{H^{f^j}(a_j(\lambda_j),\eta_k^j)}\leq \prod_{k=1}^{\floor{u}}\frac{H^{f^j}(a_j,\eta_k^j)}{H^{f^j}(a_j(\lambda_j),\eta_k^j)}
\end{align*}
since $H^{f}(a,x)$ is non-decreasing in $a$.
Therefore
\begin{align*}
Q^{(a_1,a_2)}_{m,n}(t_j\geq u)=&\frac{1}{Z_{m,n}(a_1,a_2)}\sum_{x_\centerdot \in \Pi_{m,n}}\ind_{\{x_\centerdot\geq u\}}W(a_1,a_2)(x_\centerdot)\\
\leq &\frac{Z_{m,n}(a_1(\lambda_j),a_2(\lambda_j))}{Z_{m,n}(a_1,a_2)}\prod_{k=1}^{\floor{u}}\frac{H^{f^j}(a_j,\eta_k^j)}{H^{f^j}(a_j(\lambda_j),\eta_k^j)}. 
\end{align*}
Then for all real numbers $z,r$
\begin{align}
\P\Big\{Q_{m,n}(t_j\geq u)\geq e^{-z}\Big\}\leq &  \hat{\P}\Big\{
\prod_{k=1}^{\floor{u}}\frac{H^{f^j}(a_j,\eta_k^j)}{H^{f^j}(a_j(\lambda_j),\eta_k^j)}\geq e^{-r}\Big\}\label{equation first ub prob}\\
&+\hat{\P}\Big\{ \frac{Z_{m,n}(a_1(\lambda_j),a_2(\lambda_j))}{Z_{m,n}(a_1,a_2)}\geq e^{r-z}\Big\}.\label{equation second ub prob}
\end{align}
We now split the proof into two cases.

\textbf{Case 1:}  $1\vee C_1 \kappa_N \leq u\leq \delta N$.  Let $b,\delta >0$ be small enough such that $b\delta\leq \epsilon$.  These constants will be determined through the course of the proof.  Put $\lambda_1=\frac{bu}{N}$ and $\lambda_2=-\frac{bu}{N}$.  The condition $u\leq \delta N$ guarantees that $-\epsilon\leq \lambda_2<0<\lambda_1\leq \epsilon$.  Now plug in $r=\floor{u}\Big(\psi_0^{f^j}(a_j(\lambda_j))-\psi_0^{f^j}(a_j)\Big)-\frac{\delta u^2}{N}$ and $z=\frac{\delta u^2}{N}$ to obtain 
\begin{align}
\text{RHS of }\eqref{equation first ub prob}= \hat{\P}\Big\{\sum_{k=1}^{\floor{u}}\overline{\log H^{f^j}(a_j,\eta_k^j)-\log H^{f^j}(a_j(\lambda_j),\eta_k^j)}\geq \frac{\delta u^2}{N}\Big\}\leq C\frac{N^2}{u^3}\label{inequal first conclusion ub prob}
\end{align}
by Chebyshev's inequality and the fact that $H^f(a,\eta)\sim m_f(a)$.  The constant $C$ here depends only on $(a_1,a_2)$, $\epsilon$, and $\delta$.  We will now show how to tune $b$ and $\delta$ as functions of $(a_1,a_2)$ and $\epsilon$ to get a meaningful bound on 
\begin{equation}
\begin{split} 
\eqref{equation second ub prob}=&\hat{\P}\Big\{\overline{\log Z_{m,n}(a_1(\lambda_j),a_2(\lambda_j))}-\overline{\log Z_{m,n}(a_1,a_2)}\geq \\
&\qquad \qquad \qquad\Eh\big[\log Z_{m,n}(a_1,a_2)-\log Z_{m,n}(a_1(\lambda_j),a_2(\lambda_j))\big]+ r-z\Big\}. \label{equation long}
\end{split}
\end{equation}
Since the parameters satisfy $a_1(\lambda_j)+a_2(\lambda_j)=a_3$, by Remark \ref{remark parameter DR property}, the down-right property is still satisfied for the perturbed model with parameters $\big(a_1(\lambda_j),a_2(\lambda_j)\big)$.  Using Proposition \ref{proposition variance formula} we can evaluate the right-hand side inside the above probability
\begin{align}
=& m\Big(\psi_0^{f^1}(a_1)-\psi_0^{f^1}(a_1(\lambda_j))\Big)+n\Big(\psi_0^{f^2}(a_2)-\psi_0^{f^2}(a_2(\lambda_j))\Big)+\floor{u}\Big(\psi_0^{f^j}(a_j(\lambda_j))-\psi_0^{f^j}(a_j)\Big)-2\delta\frac{u^2}{N}\nonumber\\
=&(m-N\psi_1^{f^2}(a_2))\Big(\psi_0^{f^1}(a_1)-\psi_0^{f^1}(a_1(\lambda_j))\Big)+(n-N\psi_1^{f^1}(a_1))\Big(\psi_0^{f^2}(a_2)-\psi_0^{f^2}(a_2(\lambda_j))\Big)\nonumber\\
&+N\Big[\psi_1^{f^2}(a_2)\Big(\psi_0^{f^1}(a_1)-\psi_0^{f^1}(a_1(\lambda_j))\Big)+\psi_1^{f^1}(a_1)\Big(\psi_0^{f^2}(a_2)-\psi_0^{f^2}(a_2(\lambda_j))\Big)\Big]\nonumber\\
&+\floor{u}\Big(\psi_0^{f^j}(a_j(\lambda_j))-\psi_0^{f^j}(a_j)\Big)-2\delta\frac{u^2}{N}\nonumber\\
&\geq -\kappa_N \frac{bu}{N}C' -N(\frac{bu}{N})^2 C' + u(\frac{bu}{N}) C'' -2\delta \frac{u^2}{N} \nonumber\\
&= \frac{u}{N}\Big[C''bu-C'b^2u
-2\delta u -C'b\kappa_N\Big] \label{quantity one}
\end{align}
for some positive constants $C'$ and $C''$.  This can be obtained by taking a 2nd-order Taylor expansion of the functions $\psi^{f^j}_0$, keeping in mind that $\psi^{f^j}_1>0$.  In the last inequality we also used $u\geq 1$.

Now fixing $b$ small enough followed by then fixing $\delta$ small enough we can ensure that the entire quantity \eqref{quantity one} is $\geq C'''\frac{u^2}{N}$ for some positive constant $C'''$ as long as $u\geq C_1\kappa_N$ for some positive $C_1$.
With these restrictions
\begin{align*}
\eqref{equation second ub prob}\leq & \hat{\P}\Big\{\overline{\log Z_{m,n}(a_1(\lambda_j),a_2(\lambda_j))-\log Z_{m,n}(a_1,a_2)}\geq C'''\frac{u^2}{N}\Big\}\\
\leq & \frac{N^2}{(C''')^2 u^4}\hat{\Var}\big[\log Z_{m,n}(a_1(\lambda_j),a_2(\lambda_j))-\log Z_{m,n}(a_1,a_2)\big]\\
\leq& C\frac{N^2}{u^4}\Big(\hat{\Var}\big[\log Z_{m,n}(a_1,a_2)\big]+(m+n)\frac{bu}{N}\Big)\\
\leq & C\Big(\frac{N^2}{u^4}E_{m,n}[t_j]+\frac{N^2}{u^3}\Big).
\end{align*}
The second to last and last inequalities are applications of Lemma \ref{lemma variance comparison}, Proposition \ref{proposition variance formula}, and Lemma \ref{E[t] lemma}.  Combining this result with \eqref{inequal first conclusion ub prob} finishes the first case.

\textbf{Case 2:} $1\vee C_1\kappa_N\vee \delta N\leq u$.  Take $\delta$, $\epsilon$ fixed from the first case, let $\delta_1\in (0,\delta]$, and $\epsilon_1\in (0,\epsilon]$.  The constants $\delta_1$ and $\epsilon_1$ will be determined throughout the course of the proof.  This time, put $\lambda_1= \epsilon_1,\, \lambda_2=-\epsilon_1 ,\, r=\floor{u}\big(\psi_0^{f^j}(a_j(\lambda_j))-\psi_0^{f^j}(a_j)\big)-\delta_1 u$, and $z=\delta_1 u.$ Then
\begin{align}
\eqref{equation first ub prob}= & \hat{\P}\Big\{\sum_{k=1}^{\floor{u}}\overline{\log H^{f^j}(a_j,\eta_k^j)-\log H^{f^j}(a_j(\lambda_j),\eta_k^j))}\geq \delta_1 u\Big\}.\label{quick ref 1}
\end{align}
By Remark \ref{remark - Mellin consequences} the random variables in the summation have finite exponential moments.  A large deviation estimate gives us the existence of a positive constant $C_3$ such that \eqref{quick ref 1}$\leq e^{-C_3u}$.

We now consider \eqref{equation long}. A similar analysis to that in Case 1 tells us
that the right-hand side inside of the above probability
\begin{align}
&\geq -C'\epsilon_0 \kappa_N-C'\epsilon_0^2 N+C''\epsilon_0 u-2\delta_1 u\nonumber \\
&\geq u\Big(C''\epsilon_0-\frac{C'\epsilon_0^2}{\delta}-2\delta_1\Big)-C'\epsilon_0 \kappa_N\label{quantity two}
\end{align}
for some positive constants $C'$ and $C''$. The second line follows from $u\geq \delta N$.  Now fixing $\epsilon_0$ small enough followed by then fixing $\delta_1$ small enough we can ensure that \eqref{quantity two} $\geq Cu$ for some positive constant $C$ as long as $u\geq C_1\kappa_N$ for some positive $C_1$ (here we increase the previous constant $C_1$ found in Case 1 if necessary).  With these constraints
\[
\eqref{equation second ub prob}\leq \hat{\P}\Big\{\overline{\log Z_{m,n}(a_1(\lambda_j),a_2(\lambda_j))-\log Z_{m,n}(a_1,a_2)}\geq Cu\Big\}.
\]

Since the perturbed parameters are such that the polymer environment still has the down-right property, the random variable inside the above probability can be expressed as two sums of i.i.d.\ random variables, each of which has entries with finite exponential moments.  Therefore a large deviation estimate gives the existence of a positive constant $C_3$ such that $\eqref{equation second ub prob}\leq e^{-uC_3}$.  Combining this with \eqref{quick ref 1} completes the proof.
\end{proof}
\begin{remark}
If $\epsilon>0$ is small enough such that for all $|\lambda|\leq \epsilon$, $a_1+\lambda\in D(M_{f^1})$ and $a_2-\lambda\in D(M_{f^2})$, then the constants in Proposition \ref{proposition stationary ub} can be chosen such that the conclusion also holds for the polymer environment with parameters $(a_1+\lambda, a_2-\lambda, a_3)$ for any $|\lambda|\leq \epsilon$.
\end{remark}
Using the previous proposition, we can now bound the annealed expectation of the exit points of the polymer path from the axes.
\begin{corollary}\label{corollary UB on E[t]}
Assume that all the assumptions of Proposition \ref{proposition stationary ub} hold. Then there exists a constant $C<\infty$ depending only on $(a_1,a_2),\, \beta$, and $\gamma$ such that for both $j=1,2$,
\[
E_{m,n}[t_j]\leq C N^{2/3}.
\]
\end{corollary}

\begin{proof}[Proof of Corollary \ref{corollary UB on E[t]}]
Since all of the constants $C_1$, $C_2$, $C_3$, $\delta$, $\delta_1$ determined by Proposition \ref{proposition stationary ub} depend only on $(a_1,a_2)$, $\beta$, and $\gamma$, it is sufficient to show that the constant $C$ to be determined in this proof depends only on these five constants and $\gamma$. 
Let $r\geq 1\vee C_1\gamma$.  Then $rN^{2/3}\geq 1\vee C_1 \kappa_N$.
Suppressing the $m,n$ dependence, 
\begin{align}
E[t_j]&=\int_0^\infty P(t_j\geq u)du\nonumber \\
&\leq  rN^{2/3}
+\int_{r N^{2/3}}^{r N^{2/3}\vee\delta N} P(t_j\geq u)du
+\int_{r N^{2/3}\vee\delta N}^{\infty} P(t_j\geq u)du\label{EE-C}.
\end{align}
We now bound the integrals in line \eqref{EE-C} individually.
\begin{align}
\int_{r N^{2/3}\vee\delta N}^{\infty} P(t_j\geq u)du=&\int_{r N^{2/3}\vee\delta N}^{\infty}\int_0^{e^{-\delta_1 u}} \mbb{P}\{Q(t_j\geq u)\geq x\}dx du\label{EE-C-1}\\
&+\int_{r N^{2/3}\vee\delta N}^{\infty}\int_{e^{-\delta_1 u}}^1 \mbb{P}\{Q(t_j\geq u)\geq x\}dx du.\label{EE-C-2}
\end{align}
Now
\[
\eqref{EE-C-1}\leq \int_{\delta N}^{\infty}e^{-\delta_1 u}du=\frac{1}{\delta_1}e^{-\delta_1 \delta N}\leq C \text{ for all } N\in\N.
\]
Next, by making the substitution $x=e^{-s u}$,
\[
\eqref{EE-C-2}=\int_{r N^{2/3}\vee\delta N}^{\infty}\int_{0}^{\delta_1} \P\{Q(t_j\geq u)\geq e^{-s u}\}e^{-s u} u \,ds du
\]
By Proposition \ref{proposition stationary ub}, $\P\{Q(t_j\geq u)\geq e^{-s u}\}\leq \P\{Q(t_j\geq u)\geq e^{-\delta_1 u}\} \leq 2e^{-C_3u}$ for all $u\geq r N^{2/3}\vee\delta N$ and all $0<s\leq\delta_1$.  Therefore
\[
\eqref{EE-C-2}\leq\int_{r N^{2/3}\vee\delta N}^{\infty}\int_{0}^{\delta_1}2e^{-(C_3+s) u} u\, ds du\leq \int_{\delta N}^\infty 2e^{-C_3 u}du=\frac{2}{C_3}e^{-C_3\delta N}\leq C.
\]
Combining the bounds on \eqref{EE-C-1} and \eqref{EE-C-2} shows that \eqref{EE-C} is bounded by some constant $C$ for all $N\in\N$.
We now bound the first integral of \eqref{EE-C}. Without loss of generality, assume that $r N^{2/3}<\delta N$.  Then
\begin{align}
\int_{r N^{2/3}}^{r N^{2/3}\vee\delta N} P(t_j\geq u)du
=&\int_{r N^{2/3}}^{\delta N}\int_0^{e^{-\delta \frac{u^2}{N}}} \P\{Q(t_j\geq u)\geq x\}dxdu\label{EE-B-1}\\
&+\int_{r N^{2/3}}^{\delta N}\int_{e^{-\delta \frac{u^2}{N}}}^1 \P\{Q(t_j\geq u)\geq x\}dxdu\label{EE-B-2}
\end{align}
Now
\[
\eqref{EE-B-1}\leq \int_{rN^{2/3}}^{\delta N}e^{-\delta \frac{u^2}{N}}du\leq \delta N e^{-\delta r^2 N^{1/3}}\leq \delta N e^{-\delta (1\vee C_1 \gamma )^2 N^{1/3}}\leq C.
\]
Next, by making the substitution $x=e^{-s \frac{u^2}{N}}$, 
\[
\eqref{EE-B-2}=\int_{rN^{2/3}}^{\delta N}\int_0^\delta \P\{Q(t_j\geq u)\geq e^{-s \frac{u^2}{N}}\}e^{-s \frac{u^2}{N}}\frac{u^2}{N}dsdu.
\]
By Proposition \ref{proposition stationary ub}, $ \P\{Q(t_j\geq u)\geq e^{-s \frac{u^2}{N}}\} \leq \P\{Q(t_j\geq u)\geq e^{-\delta \frac{u^2}{N}}\} \leq C_2(\frac{N^2}{u^4}E[t_j]+\frac{N^2}{u^3})$ for all $r N^{2/3}\leq u \leq \delta N$ and all $0<s\leq\delta$.  
Therefore
\begin{align*}
\eqref{EE-B-2}\leq \int_{rN^{2/3}}^\infty \int_{e^{-s \frac{u^2}{N}}}^1 C_2(\frac{N^2}{u^4}E[t_j]+\frac{N^2}{u^3})dsdu\leq C_2(\frac{E[t_j]}{3r^3}+\frac{N^{2/3}}{2r^2}).
\end{align*}
Combining the bounds on \eqref{EE-B-1}, \eqref{EE-B-2} and \eqref{EE-C}, we get the existence of a constant $C$ such that $E[t_j]\leq r N^{2/3}+C+C\frac{E[t_j]}{r^3}+C\frac{N^{2/3}}{r^2}$ for all $r\geq 1\vee C_1 \gamma $.   We can now fix $r$ large enough with respect to $C$ then rearrange to get the desired result.
\end{proof}

We can now give the proof of the upper bound of the variance of the free energy.

\begin{proof}[Proof of upper bound of Theorem \ref{theorem stationary variance bounds}]
Averaging \eqref{equation var formula 1} and \eqref{equation var formula 2} of Proposition \ref{proposition variance formula} then applying Lemma \ref{E[t] lemma} followed by Corollary \ref{corollary UB on E[t]} (recalling that $\psi_1^{f^j}(a_j) = \Var[\log R^j]$) gives
\begin{align*}
\Var[\log Z_{m,n}] &= E_{m,n}[\sum_{i=1}^{t_1} L_{R^1}(R^1_{i,0})] + E_{m,n}[\sum_{j=1}^{t_2} L_{R^2}(R^2_{0,j})]\\
&\leq C(E_{m,n}[t_1] + E_{m,n}[t_2] + 2)\\
&\leq C N^{2/3},
\end{align*}
which concludes the proof.
\end{proof}

The following corollary is obtained by combining Proposition \ref{proposition stationary ub} and Corollary \ref{corollary UB on E[t]}.
\begin{corollary}\label{corollary quenched tails}
Assume that the polymer environment is distributed as in \eqref{polymer environment distribution} and the sequence $(m,n)=(m_N,n_N)_{N=1}^{\infty}$ satisfies \eqref{direction} for some positive constant $\gamma$.
Then there exists positive constants $b_0$, $C_2$, $C_3$, $\delta$, and $\delta_1$ depending on $(a_1,a_2),\, \beta$, and $\gamma$ such that for all $N\in \N$ and $b_0\leq b\leq \delta N^{1/3}$,
\begin{align}\label{bound b small}
\P\big\{Q_{m,n}(t_j\geq bN^{2/3})\geq e^{-\delta b^2N^{1/3}}\big\} \leq \frac{2C_2}{b^3}\quad \text{for }j=1,2,
\end{align}
while for all $N\in \N$ and $b\geq b_0\vee \delta N^{1/3}$,
\begin{align}\label{bound b big}
\P\big\{ Q_{m,n}(t_j\geq bN^{2/3})\geq e^{-\delta_1 bN^{2/3}} \big\} \leq 2e^{-C_3 b N^{2/3}}\quad \text{for }j=1,2.
\end{align}
\end{corollary}

\begin{lemma}\label{lemma b cubed}
Assume that the polymer environment is distributed as in \eqref{polymer environment distribution} and the sequence $(m,n)=(m_N,n_N)_{N=1}^{\infty}$ satisfies \eqref{direction} for some positive constant $\gamma$. Then there exist constants $b_0\geq 1$ and $C>0$ depending only on $(a_1, a_2)$, $\beta$, and $\gamma$  such that for all $b\geq b_0$ and $N\in \N$, 
\begin{gather*}
P_{m,n}(t_j\geq b N^{2/3})\leq \frac{C}{b^3}\quad \text{for }j=1,2.
\end{gather*}
Therefore, for all $0< p <3$ there exists a positive constant $C'$ depending on $(a_1,a_2),\, \beta,\, \gamma$, and $p$ such that for all $N\in \N$,
\[
E_{m,n}\Big[\Big(\frac{t_j}{N^{2/3}}\Big)^p\Big]\leq C'\quad \text{for } j=1,2.
\]
\end{lemma}

\begin{proof}[Proof of Lemma \ref{lemma b cubed}]
By Corollary \eqref{corollary quenched tails} there exist positive constants $b_0,\, C_2,\,C_3,\, \delta,\, \delta_1$ with $b_0\geq 1$ such that \eqref{bound b small} holds for $b_0\leq b \leq \delta N^{1/3}$ while \eqref{bound b big} holds for $b\geq \delta N^{1/3}\vee b_0$.

We first estimate for $b\leq \delta N^{1/3}$,
\begin{align}
P_{m,n}(t_j \geq bN^{2/3}) &= \int_0^1 \P \big\{Q_{m,n}(t_j \geq bN^{2/3})\geq x\big\} dx\nonumber\\
&= \int_0^\delta \P\big\{Q_{m,n}(t_j \geq b N^{2/3}) \geq e^{-sb^2N^{1/3}} \big\} b^2N^{1/3} e^{-s b^2N^{1/3}} ds \label{b cubed integral 1}\\
&\qquad + \int_\delta^\infty \P\big\{Q_{m,n}(t_j \geq b N^{2/3}) \geq e^{-sb^2N^{1/3}} \big\} b^2N^{1/3} e^{-s b^2N^{1/3}} ds\label{b cubed integral 2}\\
&\leq \frac{2C_2}{b^3} + e^{-\delta b^2 N^{1/3}} \leq \frac{C}{b^3}\nonumber
\end{align}
for some positive constant $C$, where we made the substitution $x= e^{-sb^2N^{1/3}}$, used \eqref{bound b small} to bound the probability inside the integral of \eqref{b cubed integral 1}, and bounded the probability inside the integral of \eqref{b cubed integral 2} by 1. For $b\geq \delta N^{1/3}$, we make the substitution $x= e^{-sbN^{2/3}}$ to get
\begin{align}
P_{m,n}(t_j \geq bN^{2/3}) &= \int_0^1 \P \big\{Q_{m,n}(t_j \geq bN^{2/3})\geq x\big\} dx\nonumber\\
&= \int_0^{\delta_1} \P\big\{Q_{m,n}(t_j \geq b N^{2/3}) \geq e^{-sbN^{2/3}} \big\} bN^{2/3} e^{-s bN^{2/3}} ds \label{b cubed integral 3}\\
&\qquad + \int_{\delta_1}^\infty \P\big\{Q_{m,n}(t_j \geq b N^{2/3}) \geq e^{-sbN^{2/3}} \big\} bN^{2/3} e^{-s bN^{2/3}} ds\label{b cubed integral 4}\\
&\leq 2e^{-C_3bN^{2/3}} + e^{-\delta_1 bN^{2/3}} \leq \frac{C}{b^3}\nonumber
\end{align}
increasing the constant $C$ if necessary, where we used \eqref{bound b big} to bound the probability inside the integral of \eqref{b cubed integral 3} and bounded the probability inside the integral of \eqref{b cubed integral 4} by 1.
\end{proof}

\begin{proof}[Proof of Corollary \ref{corollary CLT}]
Let $m_1 = \floor{N\Var[\log R^2]}$. Then since $Z_{m,n}=Z_{m_1,n}\prod_{i=m_1+1}^m R^1_{i,n}$
\[N^{-\alpha/2} \overline{\log Z_{m,n}} = N^{-\alpha/2} \overline{\log Z_{m_1, n}} + N^{-\alpha/2}\sum_{i=m_1+1}^m \overline{\log R^1_{i,n}}. \]
The sequence $(m_1,n)$ satisfies \eqref{direction}. Using Chebyshev's inequality and the upper bound of Theorem \ref{theorem stationary variance bounds} shows that the term $N^{-\alpha/2}\overline{\log Z_{m_1,n}}$ converges to zero in probability. By the down-right property, the summands in the second term are i.i.d.\ with mean zero and variance $\Var[\log R^1]$. By the central limit theorem, $N^{-\alpha/2}\sum_{i=m_1+1}^m \overline{\log R^1_{i,n}}$ converges in distribution to a centered normal with variance $c_1\Var[\log R^1]$.
\end{proof}

\section{Proof of path fluctuation upper bound}\label{section PF UB}

Given $0\leq k< m$ and $0\leq l< n$, we define a partition function $Z_{m,n}^{(k,l)}$ and quenched polymer measure $Q_{m,n}^{(k,l)}$ on up-right paths from $(k,l)$ to $(m,n)$ by using the collections $\{R^1_{i,l}: k+1\leq i\leq m\}$ and $\{R^2_{k,j}: l+1\leq j \leq n\}$ as weights along the edges of the south and west boundaries of the rectangle $[k,m]\times [l,n]$ respectively, and the weights $\{(Y^1_z,Y^2_z): z\in \{k+1,\ldots, m\}\times \{l+1,\ldots, n\}\}$ for the remaining edges. 
When the original polymer environment \eqref{environment} has the down-right property, it follows that $Z_{m,n}^{(k,l)}$ has the same distribution as $Z_{m-k,n-l}$.

For an up-right path $x_\centerdot$ from $(k,l)$ to $(m,n)$, define
\[
t_1^{(k,l)} (x_\centerdot) := \max \{i: (k+i,l)\in x_{\centerdot}\},\qquad t_2^{(k,l)} (x_\centerdot) := \max\{j: (k, l+j)\in x_{\centerdot}\}.
\]

\begin{lemma}\label{lemma sub-box}
Assume that the polymer environment satisfies the down-right property. Then for all $0\leq k< m$, $0\leq l< n$, and $u\geq 0$,
\begin{align}
Q_{m,n} (v_1(l) \geq k+u) = Q_{m,n}^{(k,l)} ( t_1^{(k,l)} \geq u ) \stackrel{d}{=} Q_{m-k,n-l}( t_1 \geq u ),\label{equation sub-box 1}\\
Q_{m,n} (w_1(k) \geq l+u) = Q_{m,n}^{(k,l)} ( t_2^{(k,l)} \geq u )\stackrel{d}{=} Q_{m-k,n-l}( t_2 \geq u ).\label{equation sub-box 2}
\end{align}
\end{lemma}
\begin{proof}
For $0\leq i<m$ and $0\leq j <n$, we let 
\[
Z_{(i,j),(m,n)}:= \sum_{x_\centerdot} \prod_{k=1}^{(m-i)+(n-j)} \omega_{(x_{k-1},x_k)}
\]
denote the partition function for up-right paths from $(i,j)$ to $(m,n)$, where the sum is taken over all such paths. A decomposition shows that
\begin{align*}
Z_{m,n}^{(k,l)} &= \sum_{i=k+1}^m \left(\prod_{a=k+1}^i R^1_{a,l} \right)Y^2_{i,l+1} Z_{(i,l+1),(m,n)} + \sum_{j=l+1}^n \left(\prod_{b=l+1}^j R^2_{k,b}\right)Y^1_{k+1,j} Z_{(k+1,j),(m,n)}\\
&=\sum_{i=k+1}^m \frac{Z_{i,l}}{Z_{k,l}}Y^2_{i,l+1} Z_{(i,l+1),(m,n)} + \sum_{j=l+1}^n \frac{Z_{k,j}}{Z_{k,l}} Y^1_{k+1,j} Z_{(k+1,j),(m,n)}\\
&= \frac{Z_{m,n}}{Z_{k,l}}.
\end{align*}
We then have that for $r\in \{0,\ldots, m-k\}$
\begin{align*}
Q_{m,n}^{(k,l)}(t_1^{(k,l)} = r) &= \frac{1}{Z_{m,n}^{(k,l)}}\left(\prod_{i=1}^r R^1_{k+i,l} \right) Y^2_{k+r,l+1} Z_{(k+r,l+1),(m,n)}\\
&= \frac{1}{Z_{m,n}^{(k,l)}}\frac{Z_{k+r,l}}{Z_{k,l}} Y^2_{k+r,l+1} Z_{(k+r,l+1),(m,n)}\\
&= \frac{Z_{k+r,l} Y^2_{k+r,l+1} Z_{(k+r,l+1),(m,n)}}{Z_{m,n}}\\
&= Q_{m,n}(v_1(l) = k+r).
\end{align*}
Summing over $r\geq u$ gives the first equality in \eqref{equation sub-box 1}. The equality in distribution follows from the down-right property. An analogous argument gives \eqref{equation sub-box 2}.
\end{proof}

We can now prove the upper bound on the polymer path fluctuations under the annealed measure.

\begin{proof}[Proof of Theorem \ref{theorem stationary PF UB}]
If $\tau = 0$ this reduces to Lemma \ref{lemma b cubed}.  If $\tau \in (0,1)$ put $(k,l) = (\floor{\tau m}, \floor{\tau n})$. Then the sequence $(m-k, n-l)$ satisfies 
\begin{equation}\label{m-k n-l direction}
|m-k -M \psi_1^{f^2}(a_2)| \vee |n-l -M \psi_1^{f^1}(a_1)|\leq \gamma_0 M^{2/3}, 
\end{equation}
where $\gamma_0$ is a positive constant depending only on $\tau$ and $\gamma$ and $M= (1-\tau)N$.
We then apply Lemma \ref{lemma sub-box} to get
\begin{align}
Q_{m,n}(v_1(\floor{\tau n}) \geq \tau m + bN^{2/3}) \leq Q_{m,n}(v_1(\floor{\tau n}) \geq \floor{\tau m} + bN^{2/3})&\stackrel{d}{=} Q_{m-k,n-l}(t_1 \geq bN^{2/3}).\nonumber
\end{align} 
Applying Lemma \ref{lemma b cubed}, we get
\begin{align}\label{pf-bound1}
P_{m,n}(v_1(\floor{\tau n}) \geq \tau m + bN^{2/3}) \leq \frac{C}{b^3}.
\end{align}
The same argument in the vertical direction gives us
\begin{align}\label{pf-bound3}
P_{m,n}(w_1(\floor{\tau m}) \geq \tau n + bN^{2/3}) \leq \frac{C}{b^3}.
\end{align}

To prove the corresponding bounds for $v_0$ and $w_0$ we now let $k= \floor{\tau m-bN^{2/3}}$ and $l=\floor{\tau n -bN^{2/3}\frac{n}{m}}$. Again $(m-k,n-l)$ will satisfy \eqref{m-k n-l direction} for a different constant $\gamma_0$. Since $w_1(k)\geq\floor{\tau n}$ implies that $v_0(\floor{\tau n})\leq k$, it follows that 
\begin{align*}
Q_{m,n}(v_0(\floor{\tau n}) \leq \tau m - bN^{2/3} ) &\leq Q_{m,n}( w_1(k)\geq \floor{\tau n} )\\
&= Q_{m,n}^{(k,l)}( t_2^{(k,l)}\geq \floor{\tau n} - l )\\
&\leq Q_{m,n}^{(k,l)}( t_2^{(k,l)} \geq  C bN^{2/3})\\
&\stackrel{d}{=}Q_{m-k, n-l}(t_2\geq CbN^{2/3} ),
\end{align*}
for some constant $C$ depending on $(a_1, a_2)$, $\beta$, and  $\gamma$. Applying Lemma \ref{lemma b cubed} gives
\begin{align}\label{pf-bound2}
P_{m,n}( v_0(\floor{\tau n}) \leq \tau m - bN^{2/3})  \leq \frac{C}{b^3}.
\end{align}
An analogous argument shows that
\begin{align}\label{pf-bound4}
P_{m,n}( w_0(\floor{\tau m}) \leq \tau n - bN^{2/3}) \leq \frac{C}{b^3}.
\end{align}
Combining bounds \eqref{pf-bound1} and \eqref{pf-bound2} gives \eqref{pf-bound-v}, and \eqref{pf-bound3} with \eqref{pf-bound4} gives \eqref{pf-bound-h}, completing the proof.
\end{proof}

\section{Proof of variance and path fluctuation lower bounds}\label{section lower bounds}
\begin{proposition}\label{proposition lower bound}
Assume that the polymer environment is distributed as in \eqref{polymer environment distribution} and the sequence $(m,n)=(m_N,n_N)_{N=1}^{\infty}$ satisfies \eqref{direction} for some positive constant $\gamma$. Then there exist positive constants $c_0,\epsilon_0, N_0$ depending only on $(a_1,a_2)$, $\beta$ and $\gamma$ such that for all $N\geq N_0$,
\begin{align*}
\mathbb{P}\big( \overline{\log Z_{m,n}} \geq c_0 N^{1/3}\big)\geq \epsilon_0.
 \end{align*}
\end{proposition}

From this proposition we can obtain the lower bound of Theorem \ref{theorem stationary variance bounds}.
\begin{align*}
\Var [\log Z_{m,n}] &\geq \E\big[(\overline{\log Z_{m,n}})^2 : \overline{\log Z_{m,n}} \geq c_0N^{1/3}\big]\\
&\geq \P\big(\overline{\log Z_{m,n}} \geq c_0 N^{1/3}\big) (c_0 N^{1/3})^2  \\
&\geq \epsilon_0 c_0^2 N^{2/3}.
\end{align*}

\begin{proof}[Proof of Proposition \ref{proposition lower bound}]
Let $\epsilon>0$ be small enough such that for all $|\lambda|\leq \epsilon$, $a_1+\lambda\in D(M_{f^1})$ and $a_2-\lambda\in D(M_{f^2})$. Define
\[ \tilde{m} = \floor{m \frac{\psi_1^{f^2}(a_2-\lambda)}{\psi_1^{f^2}(a_2)}}, \qquad \tilde{n} = \floor{n \frac{\psi_1^{f^1}(a_1+\lambda)}{\psi_1^{f^1}(a_1)}}.\]
Taking Taylor expansions gives 
\begin{equation}
\begin{split}\label{equation m-m tilde}
m-\tilde{m} &= \lambda\frac{\psi_2^{f^2}(a_2)}{\psi_1^{f^2}(a_2)} m + o(\lambda)m \\
\tilde{n}-n &= \lambda\frac{\psi_2^{f^1}(a_1)}{\psi_1^{f^1}(a_1)} n + o(\lambda)n.
\end{split}
\end{equation}
Let $b$ be a fixed positive constant which will be determined through the course of the proof. Then there exists $N_0\in \N$ such that for all $N\geq N_0$, $bN^{-1/3}\leq \epsilon$.  Then with $\lambda= bN^{-1/3}$, the sequence $(\t{m},\t{n})$ satisfies
\[
|\tilde{m} - N \psi_1^{f^2}(a_2-\lambda)| \vee |\tilde{n} -N\psi_1^{f^1}(a_1+\lambda)| \leq \gamma_0 N^{2/3} 
\]
for some positive constant $\gamma_0$. By Table \ref{table psi_n^f} and \eqref{equation psi_n integral rep} in the Appendix, in each of the four basic beta-gamma models, either
$\psi_2^{f^1}(a_1)$ and $\psi_2^{f^2}(a_2)$ are both positive (inverse-beta model for certain choices of parameters and inverse-gamma model for all choices of parameters), $\psi_2^{f^1}(a_1)$ is negative and $\psi_2^{f^2}(a_2)$ is positive (gamma and beta models), or $\psi_2^{f^1}(a_1)$ is positive and $\psi_2^{f^2}(a_2)$ is non-positive  (inverse-beta model with the remaining choices of parameters). By flipping the $x$ and $y$ axes in the second case, we only need to consider the first and third cases.

For the case where $\psi_2^{f^1}$ and $\psi_2^{f^2}$ are both positive define $A_N = m-\tilde{m}$ and $B_N = \tilde{n}-n$. This case is illustrated in Figure \ref{figure case 1}. By \eqref{equation m-m tilde} and increasing $N_0$ if necessary, there exist positive constants $c_1, c_2, C_1, C_2$ such that for $N\geq N_0$,
\begin{align*}
c_1 bN^{2/3} \leq A_N \leq C_1bN^{2/3},\\
c_2 b N^{2/3}\leq B_N \leq C_2 b N^{2/3}.
\end{align*}
\begin{figure}
  \centering
  \begin{tikzpicture}
    \draw (0,0) -- (4,0) -- (4, 3) --(0,3) -- (0,0);
    \node [below] at (4,-.1) {$m$};
    \node [left] at (0,3) {$n$};
    \draw (0,0) -- (3,0) -- (3,3.5) -- (0,3.5) -- (0,0);
    \node [below] at (3,0) {$\tilde{m}$};
    \node [left] at (0,3.5) {$\tilde{n}$};
  \end{tikzpicture}
  \caption{Case 1: $\psi_2^{f_1}$ and $\psi_2^{f_2}$ are both positive.}
  \label{figure case 1}
\end{figure}
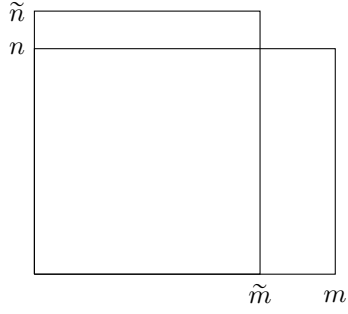

In the case where $\psi_2^{f^1}(a_1)>0$ and $\psi_2^{f^2}(a_2)\leq 0$ we define $c = \frac{1}{2}(\frac {m}{\tilde{m}} + \frac{n}{\tilde{n}})$ and let $\overline{m} = c\tilde{m}$, $\overline{n} = c\tilde{n}$. This case is illustrated in Figure \ref{figure case 2}. This $(\overline{m}, \overline{n})$ will satisfy
\[
|\overline{m} - M \psi_1^{f^2}(a_2-\lambda)| \vee |\overline{n} -M\psi_1^{f^1}(a_1+\lambda)| \leq \gamma_0 c^{1/3} M^{2/3} 
\]
where $M = cN$. A Taylor expansion gives
\[ c = 1+ \left(\frac{\psi_2^{f^2}(a_2)}{\psi_1^{f^2}(a_2)} - \frac{\psi_2^{f^1}(a_1)}{\psi_1^{f^1}(a_1)} \right) \frac{\lambda}{2} + o(\lambda)\]
and thus
\begin{align*}
m-\overline{m}  &= \frac{\lambda}{2} \left(\frac{\psi_2^{f^2}(a_2)}{\psi_1^{f^2}(a_2)} + \frac{\psi_2^{f^1}(a_1)}{\psi_1^{f^1}(a_1)} \right) m + o(N^{2/3}), \\
\overline{n} - n &= \frac{\lambda}{2} \left(\frac{\psi_2^{f^2}(a_2)}{\psi_1^{f^2}(a_2)} + \frac{\psi_2^{f^1}(a_1)}{\psi_1^{f^1}(a_1)} \right) n + o(N^{2/3}).
\end{align*}
The quantity $\frac{\psi_2^{f^2}(a_2)}{\psi_1^{f^2}(a_2)} + \frac{\psi_2^{f^1}(a_1)}{\psi_1^{f^1}(a_1)}$ is positive since $\psi_1^{f^1}$ and $\psi_1^{f^2}$ are both positive and $\psi_1^{f^2}(a_2)\psi_2^{f^1}(a_1)+\psi_1^{f^1}(a_1)\psi_2^{f^2}(a_2)>0$ by Lemma \ref{lemma -psi1 psi2 combo is positive} in the Appendix.  Letting $\overline{A} = m-\overline{m}$ and $\overline{B} = \overline{n}-n$, there exist positive constants $c_1', c_2', C_1', C_2'$ such that
\begin{align*}
c_1'bM^{2/3}\leq &\overline{A}_M \leq C_1' bM^{2/3}, \\
c_2'bM^{2/3} \leq &\overline{B}_M \leq C_2' bM^{2/3}.
\end{align*}
\begin{figure}
  \centering
  \begin{tikzpicture}
    \draw (0,0) -- (3.5,0) -- (3.5,2.75) -- (0,2.75) -- (0,0);
    \node [below] at (3.5,-.1) {$m$};
    \node [left] at (0,2.75) {$n$};
    \draw (0,0) -- (4,0) -- (4, 4.5) --(0,4.5) -- (0,0);
    \node [below] at (4,0) {$\tilde{m}$};
    \node [left] at (0,4.5) {$\tilde{n}$};
    \draw (0,0) -- (2.97,0) -- (2.97,3.34) -- (0,3.34) -- (0,0);
    \node [below] at (2.97, -.05) {$\overline{m}$};
    \node [left] at (0,3.34) {$\overline{n}$};
  \end{tikzpicture}
  \caption{Case 2: $\psi_2^{f_1}>0$ and $\psi_2^{f_2}\leq 0$.}
  \label{figure case 2}
\end{figure}
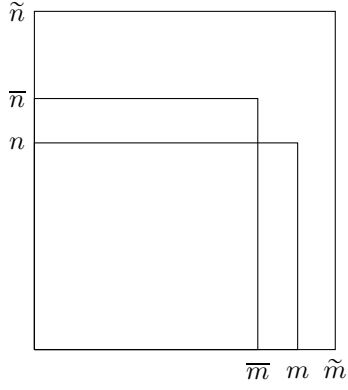
Recall that $\P^{(a_1,a_2)}$ is used to denote the probability measure on the polymer environment with parameters $a_1$ and $a_2$. Let $(\tilde{a}_1, \tilde{a}_2) =(a_1 + \lambda, a_2 - \lambda)$. Our goal is to show that
\[\mathbb{P}^{(a_1,a_2)} \big( \log Z_{m,n} \geq \E[\log Z_{m,n}] + c_0 N^{1/3}\big)\geq \epsilon_0.\]
We will do so by making estimates using the $(\tilde{a}_1, \tilde{a}_2)$ environment and then use a coupling of the two environments to transfer the results to the $(a_1,a_2)$ environment.

We would first like to show that in the $(\tilde{a}_1, \tilde{a}_2)$ environment, with high probability the quenched probability gives most of the weight to paths which exit the $x$-axis at a point of order $bN^{2/3}$. That is:  there exist constants $C_3,C$ such that, given any $\varepsilon>0$,  
\begin{align}\label{LB-estimate1}
\P^{(\tilde{a}_1, \tilde{a}_2)} \big\{ Q_{m,n} (c_1bN^{2/3}\leq t_1\leq C_3bN^{2/3}) \geq 1-\varepsilon\big\} &\geq 1-\frac{C}{b^3}
\end{align}
holds for all sufficiently large $N$.

We start by using Lemma \ref{lemma sub-box} to relate an upper bound on $t_1$ to a lower bound on $t_2$.
\begin{align*}
Q_{m,n}(t_1\leq A_N) \overset{d}{=} Q_{m,\tilde{n}}(v_1(B_N)\leq A_N) = Q_{m,\tilde{n}}(w_1(A_N)>B_N) \overset{d}{=} Q_{\tilde{m},\tilde{n}}(t_2>B_N).
\end{align*}
Using this and Corollary \ref{corollary quenched tails}, there exists $\delta>0$ such that
\begin{align*}
\P^{(\tilde{a}_1, \tilde{a}_2)}\big\{Q_{m,n}(t_1 > c_1bN^{2/3}) \geq 1-e^{-\frac{\delta}{N} B_N^2} \big\}  &\geq 
\P^{(\tilde{a}_1, \tilde{a}_2)}\big\{Q_{m,n}(t_1 > A_N) \geq 1-e^{-\frac{\delta}{N} B_N^2} \big\} \\
&= \P^{(\tilde{a}_1, \tilde{a}_2)}\big\{Q_{m,n}(t_1\leq A_N)\leq e^{-\frac{\delta}{N} B_N^2} \big\}\\
&=\P^{(\tilde{a}_1, \tilde{a}_2)}\big\{Q_{\tilde{m},\tilde{n}}(t_2> B_N)\leq e^{-\frac{\delta}{N} B_N^2} \big\}\\
&\geq 1-Cb^{-3}.
\end{align*}
This implies that
\begin{align*}
\P^{(\tilde{a}_1, \tilde{a}_2)}\big\{Q_{m,n}(t_1 \leq c_1 bN^{2/3})\geq e^{-\frac{\delta}{N} B_N^2} \big\} &\leq Cb^{-3}.
\end{align*}
Applying the upper bound directly for $C_3>C_2$, we obtain
\begin{align*}
\P^{(\tilde{a}_1, \tilde{a}_2)}\big\{ Q_{m,n} (t_1>C_3bN^{2/3})\geq e^{-\frac{\delta}{N} B_N^2} \big\} \leq Cb^{-3}
\end{align*}
for another constant $C$. Taking a union bound we put the two bounds together and get
\begin{align*}
\P^{(\tilde{a}_1, \tilde{a}_2)}\big\{Q_{m,n} (c_1 bN^{2/3} \leq t_1\leq C_3bN^{2/3})\geq 1- 2e^{-\frac{\delta}{N} B_N^2} \big\} \geq 1- Cb^{-3}.
\end{align*}
Taking $N$ large enough, we get \eqref{LB-estimate1}.

The argument for the case where we use $(\overline{m},\overline{n})$ and $\overline{A}, \overline{B}$ is unchanged, with the exception of using the lower bound in terms of the scaling parameter $M$ rather than $N$. This difference can be absorbed into the constants. 

In order to make use of this bound for the system with the original $(a_1,a_2)$ environment we create a new measure $\widecheck{\P}$ which has both $a_1$ and $\tilde{a}_1$ distributed weights along the $x$-axis and estimate the Radon-Nikodym derivative of the $(a_1, a_2)$ environment with respect to this new environment. 

Let $\widecheck{\omega}$ denote the environment that has the same weights as the $(a_1,a_2)$ environment except for the weights $R^1_{i,0}$ for $1\leq i\leq \floor{C_3bN^{2/3}}$, which will be distributed with parameter $\tilde{a}_1$. Let $\widecheck{\P}$ denote the probability measure on this environment. Then for each path $\pi$ with $c_1bN^{2/3}\leq t_1\leq C_3bN^{2/3}$, the weight of the path in the $(\tilde{a}_1, \tilde{a}_2)$ environment and the weight of the path in the $\widecheck{\omega}$ environment agree. Thus 
 \begin{align}
 Z_{m,n}(c_1bN^{2/3}\leq t_1\leq C_3bN^{2/3}) \label{LB-pp}
 \end{align}
is the same in distribution under $\P^{(\tilde{a}_1, \tilde{a}_2)}$ and $\widecheck{\P}$, where $Z_{m,n}(A):= \sum_{x_\centerdot \in A}\prod_{k=1}^{m+n}\omega_{(x_{k-1},x_k)}$.
We can now make use of the bound \eqref{LB-estimate1}. 

Using a third-order Taylor expansion, the same series of calculations which leads to inequality \eqref{quantity one} in the proof of Proposition \ref{proposition stationary ub} gives the existence of a constant $C'>0$ such that:
\begin{equation}
\begin{split}\label{LB-p0}
\E^{(\tilde{a}_1, \tilde{a}_2)}[\log Z_{m,n}]- \E^{(a_1,a_2)}[\log Z_{m,n}] &= m \left(\Psi_0^{f_1}(\tilde{a}_1) -\Psi_0^{f_1}(a_1) \right) + n\left(\Psi_0^{f_2}(\tilde{a}_2) - \Psi_0^{f_2}(a_2)\right)\\
&\geq  -\gamma bN^{1/3} C'+ 4 c_4 b^2 N^{1/3}-b^3 C' \\
&\geq c_4 b^2 N^{1/3}
\end{split}
\end{equation}
where $c_4:=\frac{1}{8}\Bigl( \psi_1^{f^2}(a_2)\psi_2^{f^1}(a_1)+\psi_1^{f^1}(a_1)\psi_2^{f^2}(a_2) \Bigr)$ is positive by Lemma \ref{lemma -psi1 psi2 combo is positive} in the Appendix. The last inequality is obtained by first fixing $b$ large enough then increasing $N_0$ if necessary. 

We now split the probability
\begin{align}
& \P^{(\tilde{a}_1, \tilde{a}_2)}\big\{ Q_{m,n} (c_1 bN^{2/3} \leq t_1\leq C_3bN^{2/3})\geq 1- \varepsilon \big\} \nonumber\\
&= \P^{(\tilde{a}_1, \tilde{a}_2)} \left\{ \frac{1}{Z_{m,n}} Z_{m,n} ( c_1bN^{2/3}\leq t_1\leq C_3bN^{2/3} )  \geq 1- \varepsilon \right\}\nonumber\\
&\leq \widecheck{\P}\left\{ Z_{m,n} (c_1bN^{2/3}\leq t_1\leq C_3bN^{2/3}) \geq (1-\varepsilon) e^{\E^{(a_1,a_2)}[\log Z_{m,n}]+\frac{1}{2} c_4 b^2 N^{1/3}}\right\} \label{LB-p1} \\
&\qquad + \P^{(\tilde{a}_1, \tilde{a}_2)}\left\{ Z_{m,n} \leq e^{\E^{(\tilde{a}_1, \tilde{a}_2)}[\log Z_{m,n}]-\frac{1}{2} c_4 b^2 N^{1/3}}\right\}. \label{LB-p2}
\end{align}
Inequality \eqref{LB-p1} comes from \eqref{LB-pp} and \eqref{LB-p0}.

For \eqref{LB-p2} we can use Chebyshev's inequality then the upper bound of the variance to get
\begin{align*}
\eqref{LB-p2} \leq \frac{C}{b^3}.
\end{align*}
Thus $\eqref{LB-p1}\geq 1-\frac{C}{b^3}$ for some new constant $C$. Let $g$ be the Radon-Nikodym derivative $d\widecheck{\P}/ d\P^{(a_1,a_2)}$. Recall that the distributions differ only on the weights along the $x$-axis up until site $\floor{C_3bN^{2/3}}$. Thus
\[ g(\omega) = \left(\frac{M_{f^1}(a_1)}{M_{f^1}(\tilde{a}_1)}\right)^{\floor{C_3bN^{2/3}}} \prod_{i=1}^{\floor{C_3bN^{2/3}}} \omega_{i,0}^{\lambda}.  \]
We can evaluate $\E^{(a_1,a_2)}[g^2]$ explicitly. Increasing $N_0$, if necessary, so that $2\lambda \leq \epsilon$,
\begin{align*}
\mathbb{E}^{(a_1,a_2)}[\omega_{i,0}^{2\lambda}]=\frac{1}{M_{f^1}(a_1)}\int_0^\infty x^{2\lambda}x^{a_1-1}f^1(x)dx=\frac{M_{f^1}(a_1+2\lam)}{M_{f^1}(a_1)}.
\end{align*}
Now
\begin{align*}
\mathbb{E}^{(a_1,a_2)}[g^2] & =\left(\frac{M_{f^1}(a_1)}{M_{f^1}(\tilde{a}_1)}\right)^{2\floor{C_3bN^{2/3}}} \prod_{i=1}^{\floor{C_3bN^{2/3}}} \mathbb{E}^{(a_1,a_2)}[\omega_{i,0}^{2\lambda}]\\
&= \left(\frac{M_{f^1}(a_1)M_{f^1}(a_1+2\lam)}{M_{f^1}(a_1+\lambda)^2}\right)^{\floor{C_3bN^{2/3}}}.
\end{align*}
Taking logarithms of both sides,
\begin{align*}
\log\mathbb{E}^{(a_1,a_2)}[g^2]=\floor{C_3 b N^{2/3}}\Bigl (\log M_{f^1}(a_1)+\log M_{f^1}(a_1+2bN^{-1/3})-2\log M_{f^1}(a_1+bN^{-1/3}) \Bigr)
\end{align*}  
Recall that $\frac{\p^2}{\p a^2} \log M_{f^1}(a)=\psi^{f^1}_1(a)>0$. Then
\begin{align*}
\lim_{N\rightarrow \infty} \log\mathbb{E}^{(a_1,a_2)}[g^2]&=C_3b \lim_{N\rightarrow \infty}
\frac{\log M_{f^1}(a_1)+\log M_{f^1}(a_1+2bN^{-1/3})-2\log M_{f^1}(a_1+bN^{-1/3})}{N^{-2/3}}\\
&=C_3b^2 \lim_{N\rightarrow \infty} \frac{\psi^{f^1}_0(a_1+2bN^{-1/3})-\psi^{f^1}_0(a_1+bN^{-1/3})}{N^{-1/3}}\\
&=C_3b^3 \psi_1^{f^1}(a_1)>0
\end{align*}
Increasing $N_0$ if necessary, there is some positive constant $c_5$ such that if $N\geq N_0$,
\[ \E^{(a_1,a_2)}[g^2]^{1/2} \leq e^{c_5 b^3}. \] 
Defining the event
\[ D=\left\{Z_{m,n} (c_1bN^{2/3}\leq t_1\leq C_3bN^{2/3}) \geq (1-\epsilon) e^{\E^{(a_1,a_2)}[\log Z_{m,n}]+\frac{1}{2} c_4 b^2 N^{1/3}}\right\},\] 
we get
\begin{align*}
1-\frac{C}{b^3}\leq \eqref{LB-p1} &= \widecheck{\P}(D)\\
&= \E^{(a_1,a_2)}[g \ind_D]\\
&\leq \big(\E^{(a_1,a_2)}[g^2]\big)^{1/2} \big(\P^{(a_1,a_2)}(D)\big)^{1/2}\\
&\leq e^{c_5 b^3} \big(\P^{(a_1,a_2)}(D)\big)^{1/2}.
\end{align*}
Thus
\begin{align*}
  \epsilon_0 := (1-\frac{C}{b^3})^2 e^{-2c_5 b^3}\leq \P^{(a_1,a_2)}(D).
\end{align*}
Finally we have that
\begin{align*}
\epsilon_0 \leq \P^{(a_1,a_2)}(D) &\leq \P^{(a_1,a_2)}\left(Z_{m,n}\geq (1-\varepsilon) e^{\E^{(a_1,a_2)}[\log Z_{m,n}]+\frac{1}{2} c_4 b^2 N^{1/3}}\right)\\
&= \P^{(a_1,a_2)}\big(\log Z_{m,n} \geq \log(1-\varepsilon) + \E^{(a_1,a_2)}[\log Z_{m,n}] + \frac{1}{2}c_4 b^2 N^{1/3}\big)\\
&\leq \P^{(a_1,a_2)}\big(\log Z_{m,n} \geq \E^{(a_1,a_2)}[\log Z_{m,n}] + c_0N^{1/3}\big).
\end{align*}
Increasing $N_0$ if necessary and taking $c_0=\frac{1}{4}c_4 b^2$ the final inequality holds for all $N\geq N_0$.
This concludes the proof.
\end{proof}

We can use the variance lower bound to obtain a lower bound on the exit points of the path from the horizontal and vertical axes.
\begin{corollary}\label{corollary pf lower bound}
Assume that the polymer environment is distributed as in \eqref{polymer environment distribution} and the sequence $(m,n)=(m_N,n_N)_{N=1}^{\infty}$ satisfies \eqref{direction} for some positive constant $\gamma$. Then there exist positive constants $c_0,\,c_1,\,N_0$ depending only on $(a_1,a_2)$, $\beta$ and $\gamma$ such that for all $N\geq N_0$,
\[
c_0\leq P_{m,n}(t_1>c_1N^{2/3} \text{ or } t_2>c_1 N^{2/3}).
\]
\end{corollary}
\begin{proof}
Averaging \eqref{equation var formula 1} and \eqref{equation var formula 2} of Proposition \ref{proposition variance formula} then applying Lemma \ref{E[t] lemma} followed by the lower bound of Theorem \ref{theorem stationary variance bounds} gives the existence of positive constants $c,\, C,\, N_0$ such that for all $N\geq N_0$
\begin{align*}
cN^{2/3}\leq\Var[\log Z_{m,n}] &= E_{m,n}[\sum_{i=1}^{t_1} L_{R^1}(R^1_{i,0})] + E_{m,n}[\sum_{j=1}^{t_2} L_{R^2}(R^2_{0,j})]\\
&\leq C(E_{m,n}[t_1+t_2] + 2).
\end{align*}
Letting $c_1:=c/6C$ and increasing $N_0$ if necessary followed by an application of the Cauchy-Schwartz inequality along with Lemma \ref{lemma b cubed} gives
\begin{align*}
3c_1\leq E_{m,n}[\frac{t_1+t_2}{N^{2/3}}]&\leq 2c_1 + E_{m,n}[\frac{t_1+t_2}{N^{2/3}}:t_1+t_2> 2c_1 N^{2/3}]\\
&\leq 2c_1 + C' P_{m,n}(t_1+t_2> 2c_1 N^{2/3})^{\frac{1}{2}}
\end{align*}
for some positive constant $C'$.  Thus
\[
c_0:=(\frac{c_1}{C'})^2\leq P_{m,n}(t_1+t_2>2c_1 N^{2/3})\leq P_{m,n}(t_1>c_1 N^{2/3} \text{ or } t_2 >c_1 N^{2/3}),
\]
which completes the proof.
\end{proof}
We now prove the path fluctuation lower bound.
\begin{proof}[Proof of \eqref{pf-lower bound}]
If $\tau=0$, this reduces to Corollary \ref{corollary pf lower bound}.  If $\tau \in (0,1)$ put $(k,l)=(\floor{\tau m},\floor{\tau n})$.  Then the sequence $(m-k,n-l)$ satisfies \eqref{direction} with a new scaling parameter $M=(1-\tau)N$.  By the down-right property and Lemma \ref{lemma sub-box}
\begin{align*}
Q_{m-k,n-l}(t_1>u \text{ or } t_2>u)&\overset{d}{=} Q_{m,n}^{(k,l)}(t_1^{(k,l)}>u \text{ or } t_2^{(k,l)}>u)\\
&=Q_{m,n}(v_1(l)>k+u \text{ or } w_1(k)>l+u)\\
&\leq Q_{m,n}(v_1(l)>\tau m+\frac{u}{2} \text{ or } w_1(k)>\tau n+\frac{u}{2})
\end{align*}
provided that $u\geq 2$.  Corollary \ref{corollary pf lower bound} applied to the sequence $(m-k,n-l)$ completes the proof.
\end{proof}

\begin{appendices}
\section{Verification of Hypothesis \ref{hypothesis CSS}}\label{appendix CSS}

\begin{lemma}\label{lemma CSS inversion}
If the function $f$ satisfies the conditions of Hypothesis \ref{hypothesis CSS} and $g(x):=f(\frac{1}{x})$ for $x\in (0,\infty)$, then $g$ also satisfies the conditions of Hypothesis \ref{hypothesis CSS}.
\end{lemma}

\begin{proof}
Note that $\supp(g)=\supp(f)^{-1}$.  Fix a compact $K\subset D(M_g)$ and let $a\in K$.  By parts (c) and (b) of Remark \ref{remark - Mellin inversion}, $\psi_0^g(a)=-\psi_0^f(-a)$ and $-K\subset D(M_f)$.  Thus there exists a positive constant $C$ depending only $-K$ such that for all $b\in -K$, \eqref{CSS 1} and \eqref{CSS 2} hold. It therefore suffices to show the following two relations hold:
\begin{align}
L^g(a,x)&=L^f(-a,\frac{1}{x}) \label{eq Lg rel to Lf} \qquad  \text{ for all } x\in \supp(g)\\
\int_0^1 \Big|\frac{\p}{\p a} L^g(a,H^g(a,p))\Big|dp &=\int_0^1 \Big|\frac{\p}{\p b} L^f(b,H^f(b,p))\Big|dp\label{eq int Lg to int Lf}
\end{align}
where the right hand side of \eqref{eq int Lg to int Lf} is evaluated at $b=-a$.

\eqref{eq Lg rel to Lf} can be proven by using $\psi_0^g(a)=-\psi_0^f(-a)$ and making the substitution $y\mapsto \frac{1}{y}$ in the first integral appearing in \eqref{eq Lf integral form}.  

\eqref{eq int Lg to int Lf} will now follow from \eqref{eq Lg rel to Lf} and
\[
H^g(a,1-p)=\frac{1}{H^f(-a,p)} \qquad \text{ for all } p\in (0,1).
\]
To see that this equality holds, let $X\sim m_g(a)$ and $x>0$.  Using part (a) of Remark \ref{remark - Mellin inversion}
\begin{align}
F^g(a,x)=\P(X\leq x)=\P(X^{-1}\geq x^{-1})=1-\P(X^{-1}<x^{-1})=1-F^f(-a,x^{-1}).\label{eq Fg to Ff}
\end{align}
Fix $p\in (0,1)$ and recall the definition of $H^{\bullet}$, \eqref{eq def Hf}. Note that $H^f(-a,p)$ and $H^g(a,1-p)$ lie in $\supp(f)$ and $\supp(g)=\supp(f)^{-1}$ respectively.  Plugging $x=H^g(a,1-p)$ into \eqref{eq Fg to Ff} gives
\[
1-p=F^g(a,H^g(a,1-p))=1-F^f\big(-a,\frac{1}{H^g(a,1-p)}\big).
\]
Rearranging yields 
\[
F^f\big(-a,\frac{1}{H^g(a,1-p)}\big)=p=F^f(-a,H^f(-a,p)).
\]
Since $x\mapsto F^f(-a,x)$ is one-to-one on $\supp(f)$ we have the desired result.
\end{proof}

\begin{lemma}\label{lemma f are CSS}
Each of the functions $f$ in Table \ref{table f} satisfy Hypothesis \ref{hypothesis CSS}.
\end{lemma}
\begin{proof}
Fix $b>0$.  By Lemma \ref{lemma CSS inversion} it suffices to show the three functions
\[f(x)=(1-x)^{b-1}, \qquad\qquad f(x)=(1-x)^{b-1}\ind_{\{0<x<1\}}, \qquad\qquad f(x)=\big(\frac{x}{x+1}\big)^b
\]
satisfy the conditions of Hypothesis \ref{hypothesis CSS}.  In \cite{S2012} (equation 3.30 and the computation following equation 4.7), Sepp\"{a}l\"{a}inen showed that the function $f(x)=e^{-bx}$ satisfies these conditions.

We will write $C_0(a),C_1(a),\dots$ to indicate the positive constants $C_k(a)$ have a continuous dependence on $a$.  We claim it is sufficient to show that if $f(x)=(1-x)^{b-1}\ind_{\{0<x<1\}}$ or $f(x)=(\frac{x}{x+1})^b$, then for all $x\in \supp(f)$ the following three bounds hold:
\begin{align}
L^f(a,x)&\leq C_0(a)(1+|\log x|) \label{bound Lf 1}\\
|x\frac{f'(x)}{f(x)}|L^f(a,x)&\leq C_1(a)(1+|\log x|)\label{bound Lf 2}\\
|G^f(a,x)|&\leq C_2(a)(1+(\log x)^2)\label{bound Gf}
\end{align}
where
\begin{align}
G^f(a,x):&=\frac{x^{-a}}{f(x)} \int_0^x
  (\psi_1^f(a) + \psi_0^f(a)\log y - (\log y)^2) y^{a-1} f(y) dy\label{eq def Gf} \\
  &= -\frac{x^{-a}}{f(x)} \int_x^\infty
  (\psi_1^f(a) + \psi_0^f(a)\log y - (\log y)^2) y^{a-1} f(y) dy.\nonumber
\end{align}
Note that the second equality in the definition of $G^f(a,x)$ follows from the definitions of $\psi_0^f(a)$ and $\psi_1^f(a)$ in part (c) of Remark \ref{remark - Mellin consequences}.
\eqref{bound Lf 1} clearly implies \eqref{CSS 1}.  
To show \eqref{CSS 2} is satisfied, using \eqref{equation partial H}, we calculate
\begin{align}
\frac{\p}{\p a} L^f(a, H^f(a, p)) &= \frac{\p L}{\p a}(a, H^f(a,p)) + \frac{\p}{\p a} H^f(a,p) \frac{\p L}{\p x}(a, H^f(a,p))\nonumber\\
&= \left.\left(\frac{\p L}{\p a}(a, x) + x L^f(a,x) \frac{\p L}{\p x}(a, x)\right)\right|_{x=H^f(a,p)}. \nonumber
\end{align}

Since
\begin{align*}
\frac{\p L}{\p a}(a, x) + x L^f(a,x) \frac{\p L}{\p x}(a, x) =&(\psi_0^f(a)-2\log{x})L^f(a,x) - a L^f(a,x)^2\\
&+ G^f(a,x)- x\frac{f'(x)}{f(x)}L^f(a,x)^2,
\end{align*}
the conditions \eqref{bound Lf 1}, \eqref{bound Lf 2}, and \eqref{bound Gf} imply the existence of a positive constant $C_3(a)$ such that for all $x\in \supp(f)$

\begin{equation*}
\Big | \frac{\p L}{\p a}(a,x) +  x L^f(a,x) \frac{\p L}{\p x}(a,x) \Big|\leq C_3(a)\big(1+(\log x)^2\big).
\end{equation*}
Condition \eqref{CSS 2} now follows from
\begin{align*}
\int_0^1\left| \frac{\p}{\p a} L^f(a, H^f(a, p))\right| dp  \leq C_3(a)\int_0^1\big(1 + (\log H^f(a, p))^2\big) dp = C_3(a)\big(1+\psi_1^f(a)+(\psi^f_0(a))^2\big)<\infty.
\end{align*}
The last equality is justified by parts (a) and (c) of Remark \ref{remark - Mellin consequences} along with the fact that $H^f(a,\eta)\sim m_f(a)$ when $\eta$ is uniformly distributed on $(0,1)$.

We first show \eqref{bound Lf 1}, \eqref{bound Lf 2} and \eqref{bound Gf} for the case \textbf{$f(x)=(1-x)^{b-1}\ind_{\{0<x<1\}}$}.  Let $a\in D(M_f)=(0,\infty)$.  Then there exists some positive constant $C_4(a)$ such that 
\begin{align*}
\big|\psi_0^f(a)-\log y\big|y^{a-1}f(y)&\leq \begin{cases}
C_4(a)(1-\log y)y^{a-1} &\mbox{ if } 0<y<\frac{1}{2}\\
C_4(a)(1-y)^{b-1} &\mbox{ if }  \frac{1}{2}\leq y<1
\end{cases} \qquad \text{ and }\\
\big| \psi_1^f(a) + \psi_0^f(a)\log y - (\log y)^2\big | y^{a-1} f(y)&\leq 
\begin{cases}
C_4(a)\big(1+(\log y )^2\big)y^{a-1} &\mbox{ if } 0<y<\frac{1}{2}\\
C_4(a)(1-y)^{b-1} &\mbox{ if } \frac{1}{2}\leq y <1
\end{cases}
\end{align*}
Since $a>0$,  \eqref{eq Lf integral form} and \eqref{eq def Gf} give: for $0<x<\frac{1}{2}$
\begin{align}
  L^f(a,x) &\leq \frac{2^b C_4(a)}{x^a}\int_0^x (1 -\log y) y^{a-1} dy \leq C_0(a) (1+|\log x|) \qquad \text{ and } \label{equation L^f beta bound 1}\\
\big|G^f(a,x)\big|&\leq \frac{2^b C_4(a)}{x^a}\int_0^x \big(1+(\log y)^2\big) y^{a-1}dy\leq C_2(a)\big(1+(\log x)^2\big).\nonumber
\end{align}
Similarly, the secondary expressions in \eqref{eq Lf integral form} and \eqref{eq def Gf}  give: for $1/2\leq x <1$
\begin{align}
L^f(a,x)&\leq \frac{2^a C_4(a)}{(1-x)^{b-1}}\int_x^1 (1-y)^{b-1} dy \leq C_0(a) (1-x)\qquad \text{ and } \label{equation L^f beta bound 2}\\
\big|G^f(a,x)\big|&\leq \frac{2^a C_4(a)}{(1-x)^{b-1}}\int_x^1 (1-y)^{b-1}dy\leq C_2(a)(1-x)\nonumber
\end{align}
where we increased $C_0(a)$ and $C_2(a)$ if necessary.
Thus the bounds \eqref{bound Lf 1} and \eqref{bound Gf} hold.
Moreover, by \eqref{equation L^f beta bound 1} and \eqref{equation L^f beta bound 2}
\[
\big| x\frac{f'(x)}{f(x)}\big|L^f(a,x)=|b-1|\frac{x}{1-x}L^f(a,x)\leq 
\begin{cases}
C_1(a) (1+|\log x|) &\mbox{ if } 0\leq x<\frac{1}{2}\\
C_1(a) (1-x) &\mbox{ if } \frac{1}{2}\leq x<1
\end{cases}
\]
proving the bound \eqref{bound Lf 2}.

We now consider the case \textbf{$f(x)=(\frac{x}{x+1})^b$}. Let $a\in D(M_f)=(-b,0)$.  Then
\begin{align*}
\big|\psi_0^f(a)-\log y\big|y^{a-1}f(y)&\leq \begin{cases}
C_4(a)(1-\log y)y^{a+b-1} &\mbox{ if } 0<y<1\\
C_4(a)(1+\log y)y^{a-1} &\mbox{ if } y\geq 1
\end{cases}\qquad \text{ and }\\
\big| \psi_1^f(a) + \psi_0^f(a)\log y - (\log y)^2\big | y^{a-1} f(y)&\leq 
\begin{cases}
C_4(a)\big(1+(\log y )^2\big)y^{a+b-1} &\mbox{ if } 0<y<1\\
C_4(a)(1+(\log y)^2) y^{a-1} &\mbox{ if } y\geq 1
\end{cases}.
\end{align*}
Since $a+b>0$, \eqref{eq Lf integral form} and \eqref{eq def Gf} give: for $0<x<1$
\begin{align*}
  L^f(a,x) &\leq \frac{2^b C_4(a)}{x^{a+b}}\int_0^x (1 -\log y) y^{a+b-1} dy \leq C_0(a) (1+|\log x|) \qquad \text{ and }\\
\big|G^f(a,x)\big|&\leq \frac{2^b C_4(a)}{x^{a+b}}\int_0^x \big(1+(\log y)^2\big) y^{a+b-1}dy\leq C_2(a)\big(1+(\log x)^2\big).
\end{align*}
Similarly, since $a<0$, the secondary expressions in \eqref{eq Lf integral form} and \eqref{eq def Gf} give: for $x\geq 1$
\begin{align*}
L^f(a,x)&\leq \frac{2^b C_4(a)}{x^a}\int_x^\infty (1+\log y)y^{a-1} dy \leq C_0(a) (1+|\log x|)\\
\big|G^f(a,x)\big|&\leq \frac{2^b C_4(a)}{x^a}\int_x^\infty (1+(\log y)^2) y^{a-1}dy\leq C_2(a)(1+(\log x)^2)
\end{align*}
where we increased $C_0(a)$ and $C_2(a)$ if necessary.
Thus the bounds \eqref{bound Lf 1} and \eqref{bound Gf} hold.
Since $|x\frac{f'(x)}{f(x)}|=b\frac{1}{x+1}\leq b$, \eqref{bound Lf 1} implies \eqref{bound Lf 2} completing the proof.
\end{proof}

\section{Lemmas used in Section \ref{section Mellin} and Section \ref{section variance UB}} \label{appendix misc}

\begin{lemma}\label{lemma finite square moments}
Assume the polymer environment is such that $\log R^1$, $\log R^2$, $\log Y^1$, and $\log Y^2$ have finite second moments. Then $\E[(\log Z_x)^2] <\infty$ for any $x\in \Z^2_+$.
\end{lemma}
\begin{proof}
Since $\log Z_{k,0} = \sum_{i=1}^k R^1_{i,0}$ and  $\log Z_{0,\ell} = \sum_{j=1}^\ell \log R^2_{0,j}$, $\log Z_x$ has finite second moment for each $x\in \Z^2_+ \setminus \N^2$. If $x \in \N^2$, the recursion \eqref{equation Z recursion} implies that
\[ (\log Y^1_x + \log Z_{x-\alpha_1}) \wedge (\log Y^2_x + \log Z_{x-\alpha_2}) \leq \frac{\log Z_x}{2} \leq (\log Y^1_x + \log Z_{x-\alpha_1}) \vee (\log Y^2_x + \log Z_{x-\alpha_2}). \]
Thus 
\[(\log Z_x)^2 \leq 4(\log Y^1_x + \log Z_{x-\alpha_1})^2 + 4(\log Y^2_x + \log Z_{x-\alpha_2})^2.\]
Since $\log Y^1$ and $\log Y^2$ have finite second moments, an inductive argument finishes the proof.
\end{proof}
\begin{lemma}\label{lem cov calc}
Suppose $f_k:(0,\infty) \to [0,\infty)$ for $k= 1,\ldots, r$ and $a_0<a<a_1$ are real numbers such that $[a_0,a_1]\subset\bigcap_{k=1}^r D(M_{f_k})$.  Suppose we have a collection of independent random variables $\{X_k\}_{k=1}^r$ where $X_k\sim m_{f_k}(a)$ for all $ 1\leq k\leq r$.  Let $\E^{a}$ be the expectation corresponding to the product measure induced by $\{X_k\}_{k=1}^r$. 

Let $S=\sum_{k=1}^r \log{X_k}$ and $A:\R^r\rightarrow\R$ be a measurable function such that $\E^{a}[A(X_1,\ldots, X_r)^2]<\infty$ for all $a\in [a_0,a_1]$.  Then
\[
\frac{\p}{\p a}\E^{a}[A(X_1,\ldots, X_r)]=\Cov^a(A(X_1,\ldots, X_r),S) \quad \text{ for all } \quad a\in(a_0,a_1) 
\]
and
$(a_0,a_1)\ni a\mapsto \frac{\p}{\p a}\E^{a}[A(X_1,\ldots, X_r)]$ is continuous.
\end{lemma}
\begin{proof}
The joint density of $(\log X_1 ,\log X_2 ,\dots,\log X_r )$ is given by 
\[
g(x_1,\dots,x_r)=\frac{e^{a\sum_{k=1}^rx_k}}{\prod_{k=1}^rM_{f_k}(a)}\prod_{k=1}^rf_k(e^{x_k}).
\]
Thus the density of $S$ is given by 
\begin{align}
h_a(s)=\frac{e^{as}}{\prod_{k=1}^rM_{f_k}(a)}\int_{\mathbb{R}^{r-1}}f_1(e^{x_1})f_2(e^{x_2-x_1})\dots f_r(e^{s-x_{r-1}})dx_1,\dots,x_{r-1}\label{eq:h}
\end{align}
Therefore the joint density of $(\log X_1 ,\log X_2 ,\dots,\log X_r )$ given that $S=s$ is 
\[
\frac{g(x_1,\dots,x_r)\ind_{\{\sum_{k=1}^rx_k=s\}}}{h_a(s)}=\frac{\prod_{k=1}^rf_k(e^{x_k}) \ind_{\{\sum_{k=1}^rx_k=s\}}}{\int_{\mathbb{R}^{r-1}}f_1(e^{x_1})f_2(e^{x_2-x_1})\dots f_r(e^{s-x_{r-1}})dx_1,\dots,x_{r-1}},
 \]
which has no $a$ dependence.  Thus
\begin{align*}
\frac{\p}{\p a}\mathbb{E}^a[A(X_1,\ldots, X_r)] &=\frac{\p}{\p a}\int_{\mathbb{R}}\mathbb{E}^a[A(X_1,\ldots, X_r)|S=s]h_a(s)ds\\
\ &=\int_{\mathbb{R}}\mathbb{E}^a[A(X_1,\ldots, X_r)|S=s]\frac{\p}{\p a}h_a(s)ds\\
\ &=\int_{\mathbb{R}}\mathbb{E}^a[A(X_1,\ldots, X_r)|S=s]h_a(s)\big(s-\sum_{k=1}^r\frac{\p}{\p a}\log M_{f_k}(a) \big) ds\\
\ &=\text{Cov}^a(A(X_1,\ldots, X_r),S).
\end{align*}
The last equality comes from $\mathbb{E}[S]=\sum_{k=1}^r\mathbb{E}[\log X_k ]=\sum_{k=1}^r\frac{\p}{\p a}\log M_{f_k}(a) $, by part (a) of Remark \ref{remark - Mellin consequences}.
The interchanging of the derivative and the integral is justified by the bound 
\begin{align}
\int_{\mathbb{R}}\mathbb{E}[\left|A(X_1,\ldots, X_r)\right|\big|S=s]\sup_{a\in[a_0,a_1]}\left|\frac{\p}{\p a}h_a(s)\right|ds<\infty. \label{eq:sup deriv}
\end{align}
Once we show that there is a constant $C$ depending only on $a_0$ and  $a_1$ such that 
\begin{align}
\sup_{a\in[a_0,a_1]}\left|\frac{\p}{\p a}h_a(s)\right|\leq C(1+|s|)(h_{a_0}(s)+h_{a_1}(s)) \label{eq:sup bound}
\end{align}
we will have the bound \eqref{eq:sup deriv} since
\begin{align*}
\int_{\mathbb{R}}\mathbb{E}[\left|A(X_1,\ldots, X_r)\right|\big|S=s](1+|s|)h_{a_j}(s)ds=\mathbb{E}^{a_j}[|A(X_1,\ldots, X_r)|(1+|S|)]\\
\leq \mathbb{E}^{a_j}[A(X_1,\ldots, X_r)^2]^\frac{1}{2}\mathbb{E}^{a_j}[(1+|S|)^2]^\frac{1}{2}.
\end{align*}
The last expression is finite since $\E^{a_j}[A(X_1,\ldots, X_r)^2]<\infty$ by assumption, and $S$ is a finite sum of independent random variables each of which has finite exponential moments, by part (a) of Remark \ref{remark - Mellin consequences}.
Notice that the bound \eqref{eq:sup deriv} also implies that $a\mapsto \frac{\p}{\p a}\E^a[A(X_1,\ldots, X_r)]$ is continuous.
All that is left to do is verify the bound \eqref{eq:sup bound}.  To accomplish this, notice that equation \eqref{eq:h} implies that $\frac{\p}{\p a}\log{h_a(s)}=s-\E^a[S]$. So 
\[\sup_{a\in[a_0,a_1]}\left|\frac{\p}{\p a}h_a(s)\right|\leq C_1(1+|s|)\sup_{a\in[a_0,a_1]}h_a(s)
\]
 where $C_1:=\max(1,\sup_{a\in[a_0,a_1]}|E^a[S]|)$.  Thus it suffices to show that $\sup_{a\in[a_0,a_1]}h_a(s)\leq C_2\Bigl(h_{a_0}(s)+h_{a_1}(s)\Bigr)$ for some constant $C_2$ independent of $s$. Since  By part (c) of Remark \ref{remark - Mellin consequences}, $a\mapsto \E^a[S]$ is an increasing function.  Therefore for all $ s\leq\E^{a_0}[S]$, the function $a\mapsto h_a(s)$ is non-increasing on $[a_0,a_1]$.  Thus 
\[
\sup_{a\in[a_0,a_1]}h_a(s)\leq h_{a_0}(s) \text{ for all } s\leq \E^{a_0}[S].
\]
On the other hand, if $s> \E^{a_0}[S]$, then $\frac{\p}{\p a}\log \Bigl(h_a(s)\exp{(a(\E^{a_1}[S]-\E^{a_0}[S])})\Bigr)=s-\E^a[S]+\E^{a_1}[S]-\E^{a_0}[S]> 0$ for all $a\in[a_0,a_1]$.  Thus for all $s> \E^{a_0}[S]$, $a\mapsto h_a(s)\exp{\Bigl(a(\E^{a_1}[S]-\E^{a_0}[S])\Bigr)}$ is increasing on the interval $[a_0,a_1]$.  Therefore
\[
\sup_{a\in[a_0,a_1]}h_a(s)\leq C_3h_{a_1}(s) \text{ for all } s> \E^{a_0}[S]
\]
where $C_3=\exp{\Bigl((a_1-a_0)(\E^{a_1}[S]-\E^{a_0}[S]})\Bigr)$.  We now get the desired result with $C_2=1+C_3$.
\end{proof}

\begin{lemma}\label{lemma Q(t_j>k)}
Assume that the polymer environment is distributed as in \eqref{polymer environment distribution} and let $\epsilon$ be small enough such that for all $|\lambda|\leq\epsilon$, $a_1+\lambda\in D(M_{f^1})$ and $a_2-\lambda\in D(M_{f^2})$. Let $(m,n)\in \N^2$ and $k\in \N$. Then, with notation as in \eqref{equations coupling}, $Q_{m,n}^{(a_1+\lambda,a_2-\lambda)}(t_1\geq k)$ is stochastically non-decreasing in $\lambda$ and $Q_{m,n}^{(a_1+\lambda,a_2-\lambda)}(t_2\geq k)$ is stochastically non-increasing in $\lambda$.
\end{lemma}

\begin{proof}

\begin{align}\label{equation Q partial}
\frac{\p}{\p b_i} Q^{(b_1, b_2)}_{m,n} (t_j\geq k) &= \frac{\p}{\p b_i} \left(\frac{1}{Z_{m,n}(b_1,b_2)} \sum_{x_\centerdot\in \Pi_{m,n}}\ind_{\{t_j\geq k\}}W(b_1,b_2)(x_\centerdot) \right)
\end{align}
If $i\neq j$, the sum in \eqref{equation Q partial} has no $b_i$ dependence, so
\begin{align*}
\frac{\p}{\p b_i} Q^{(b_1, b_2)}_{m,n} (t_j\geq k) &= \frac{-1}{(Z_{m,n}(b_1,b_2))^2}\left(\frac{\p}{\p b_i}Z_{m,n}(b_1,b_2)\right) \sum_{x_\centerdot\in \Pi_{m,n}}\ind_{\{t_j\geq k\}}W(b_1,b_2)(x_\centerdot),
\end{align*}
which is non-positive by \eqref{equation deriv logZ}. If $i=j$, then by \eqref{equation deriv W} and \eqref{equation deriv logZ},
\begin{align*}
\frac{\p}{\p b_i} Q^{(b_1, b_2)}_{m,n} (t_i\geq k) &= \frac{\sum_{x_\centerdot\in \Pi_{m,n}}\ind_{\{t_i\geq k\}}\frac{\p}{\p b_i}W(b_1,b_2)(x_\centerdot)}{Z_{m,n}(b_1,b_2)} -\left(\frac{\p}{\p b_i} \log Z_{m,n}(b_1,b_2)\right) \frac{\sum_{x_\centerdot\in \Pi_{m,n}}\ind_{\{t_i\geq k\}}W(b_1,b_2)(x_\centerdot)}{Z_{m,n}(b_1,b_2)}\\
&= \Cov^{Q_{m,n}^{(b_1,b_2)}}\big(\sum_{k=1}^{t_i} L^{f^i}(b_i, H^{f^i}(b_i,\eta^i_k)), \ind_{\{t_i\geq k\}}\big),
\end{align*}
which is non-negative.
\end{proof}

\section{Properties of \texorpdfstring{$\psi_n^f$}{psi}}\label{appendix psi}

\begin{figure}[H]
\[
	\begin{array}{|l||c|c|}
		\hline
		\text{Model} & \psi_n^{f^1}(a_1) & \psi_n^{f^2}(a_2)  \\ \hline \hline
		\text{IG} & (-1)^{n+1}(\Psi_n(\mu-\theta) - \delta_{n,0} \log \beta) & (-1)^{n+1}(\Psi_n(\theta) - \delta_{n,0}\log \beta ) \\ \hline 
		\text{G} & \Psi_n(\mu+\theta) - \delta_{n,0} \log \beta & (-1)^{n+1}(\Psi_n(\theta) - \Psi_n(\mu+\theta)) \\ \hline
		\text{B} & \Psi_n(\mu+\theta) - \Psi_n(\mu+\theta+\beta) & (-1)^{n+1}(\Psi_n(\theta) - \Psi_n(\mu+\theta)) \\ \hline
		\text{IB} & (-1)^{n+1}(\Psi_n(\mu-\theta) - \Psi_n(\mu-\theta + \beta))  & \Psi_n(\mu-\theta + \beta) + (-1)^{n+1} \Psi_n(\theta) \\ \hline 
	\end{array}
\]
\caption{$\psi_n^{f}$ functions for each of the four basic beta-gamma models.}\label{table psi_n^f}
\end{figure}

By \cite{AS1964} (p.260 line 6.4.1) the polygamma function of order $n$, $\Psi_n(x)= \frac{\p^{n+1}}{\p x^{n+1}} \log \Gamma(x)$, has integral representation
\begin{equation}\label{equation psi_n integral rep}
\Psi_n(x)=(-1)^{n+1} \int_0^\infty \frac{t^n e^{-xt}}{1-e^{-t}}dt.
\end{equation}

\begin{lemma} \label{lemma psi are log-convex}
For any $n\in \N$, the map $a\mapsto \frac{\Psi_{n+1}(a)}{\Psi_n(a)}$ is strictly increasing on $(0,\infty)$.
\end{lemma}

\begin{proof}
Fix $n\in \N$ and $a\in (0,\infty)$.  We will show that $\frac{\p^2}{\p a^2}\log |\Psi_n(a)|>0$.  

After substituting $y=e^{-t}$ in \eqref{equation psi_n integral rep} we get
\[
|\Psi_n(a)|=\int_0^\infty y^{a-1}f(y)dy=M_f(a)
\]
where $f(y):=\frac{(-\log y)^n}{1-y}\ind_{\{0<y<1\}}$.  Note that $D(M_f)=(0,\infty)$. Now given a random variable $X\sim m_f(a)$, by part (c) of Remark \ref{remark - Mellin consequences}
\[
\frac{\p^2}{\p a^2}\log |\Psi_n(a)|=\frac{\p^2}{\p a^2}\log M_f(a)=\Var[\log X]>0
\]
since $X$ is non-degenerate.  
\end{proof}

\begin{lemma}\label{lemma -psi1 psi2 combo is positive}
Assume the polymer environment is distributed as in \eqref{polymer environment distribution}.  Then \[
\psi_1^{f^1}(a_1)\psi_2^{f^2}(a_2)+\psi_1^{f^2}(a_2)\psi_2^{f^1}(a_1)>0.
\]
\end{lemma}

\begin{proof}
Recall that $\psi_1^{f^j}$ are always positive and by \eqref{equation psi_n integral rep} $\Psi_n$ has sign $(-1)^{n+1}$ throughout $(0,\infty)$.

For the inverse-gamma model \eqref{model-IG} with fixed constants $\beta>0$ and $\mu>\theta>0$, Table \ref{table psi_n^f} implies that $\psi_2^{f^j}(a_j)>0$ for $j=1,2$.
The conclusion follows immediately.

For the gamma model \eqref{model-G} with fixed positive constants $\beta$, $\mu$,  and $\theta$,  by Table \ref{table psi_n^f}

\begin{align}
\psi_1^{f^1}(a_1)\psi_2^{f^2}(a_2)+\psi_1^{f^2}(a_2)\psi_2^{f^1}(a_1)=-\Psi_1(\theta+\mu)\Psi_2(\theta)+\Psi_1(\theta)\Psi_2(\theta+\mu).\nonumber
\end{align}
The quantity on the right hand side is positive if and only if 
\[
\frac{\Psi_2(\theta+\mu)}{\Psi_1(\theta+\mu)}>\frac{\Psi_2(\theta)}{\Psi_1(\theta)}
\]
which holds true by Lemma \ref{lemma psi are log-convex} with $n=1$.

For the beta model \eqref{model-B} with fixed positive constants $\beta$, $\mu$,  and $\theta$, using Table \ref{table psi_n^f}, 
\begin{align}
\psi_1^{f^1}(a_1)\psi_2^{f^2}(a_2)+\psi_1^{f^2}(a_2)\psi_2^{f^1}(a_1)&>0 &  &\Leftrightarrow\nonumber\\
\frac{\psi_2^{f^1}(a_1)}{\psi_1^{f^1}(a_1)}&>-\frac{\psi_2^{f^2}(a_2)}{\psi_1^{f^2}(a_2)} & &\Leftrightarrow\nonumber\\
\frac{\Psi_2(\theta+\mu+\beta)-\Psi_2(\theta+\mu)}{\Psi_1(\theta+\mu+\beta)-\Psi_1(\theta+\mu)}&>\frac{\Psi_2(\theta+\mu)-\Psi_2(\theta)}{\Psi_1(\theta+\mu)-\Psi_1(\theta)}.\label{inequal temp-1}
\end{align}

By Cauchy's mean value theorem there exist constants $\theta<\xi_1<\theta+\mu<\xi_2<\theta+\mu+\beta$ such that the left and right-hand sides of \eqref{inequal temp-1} equal $\frac{\Psi_3(\xi_2)}{\Psi_2(\xi_2)}$ and $\frac{\Psi_3(\xi_1)}{\Psi_2(\xi_1)}$ respectively.  Lemma \ref{lemma psi are log-convex} with $n=2$ now gives \eqref{inequal temp-1}.

For the inverse-beta model \eqref{model-IB} with fixed  constants $\beta>0$ and  $\mu>\theta>0$, by Table \ref{table psi_n^f}, $\psi_2^{f^1}(a_1)>0$, $\psi_1^{f^2}(a_2)>\psi_1(-\theta+\mu+\beta)$, and $\psi_2^{f^2}(a_2)>\psi_2(-\theta+\mu+\beta)$.  Therefore
\begin{align*}
\psi_1^{f^1}(a_1)\psi_2^{f^2}(a_2)+\psi_1^{f^2}(a_2)\psi_2^{f^1}(a_1)>& \psi_1^{f^1}(a_1)\Psi_2(-\theta+\mu+\beta)+\Psi_1(-\theta+\mu+\beta)\psi_2^{f^1}(a_1)\\
&=\Psi_1(-\theta+\mu)\Psi_2(-\theta+\mu+\beta)-\Psi_1(-\theta+\mu+\beta)\Psi_2(-\theta+\mu).
\end{align*}
Letting $x=-\theta+\mu$, the last line is positive if and only if 
\[
\frac{\Psi_2(x+\beta)}{\Psi_1(x+\beta)}>\frac{\Psi_2(x)}{\Psi_1(x)}
\]
which holds true by Lemma \ref{lemma psi are log-convex} with $n=1$.
\end{proof}

\end{appendices}

\end{document}